\documentclass[11pt,reqno]{amsart}
\usepackage[]{amsmath,amssymb,amsfonts,latexsym,amsthm,mathabx}
\usepackage{amssymb,bbm,mathrsfs,tikz,stmaryrd}
\usetikzlibrary{patterns,arrows.meta}
\usepackage[normalem]{ulem}
\usepackage{graphicx}
\usepackage[font=small]{caption}
\usepackage[noadjust]{cite}

\numberwithin{equation}{section}
\usepackage[colorlinks=true, pdfstartview=FitV, linkcolor=blue,
  citecolor=blue, urlcolor=blue,pagebackref=false]{hyperref}
\usepackage[shortlabels]{enumitem}  
\usepackage[capitalise]{cleveref}
\usepackage[colorinlistoftodos]{todonotes}
\presetkeys{todonotes}{inline, color=green}{}
\usepackage{comment}

\newtheorem{theorem}{Theorem}[section]
\newtheorem*{theorem*}{Theorem}
\newtheorem{lemma}[theorem]{Lemma}

\newtheorem{proposition}[theorem]{Proposition}

\newtheorem{remark}[theorem]{Remark}
\newtheorem*{remark*}{Remark}

\newtheorem{definition}[theorem]{Definition}
\newtheorem*{definition*}{Definition}

\newtheorem*{question*}{Question}
\newtheorem*{example*}{Example}
\newtheorem*{examples*}{Examples}
\newcommand{\abbr}[1]{{\sc{\lowercase{#1}}}}

\newcommand{\E}{\mathbb E}
\newcommand{\N}{\mathbb N}
\renewcommand{\P}{\mathbb P}
\newcommand{\Q}{\mathbb Q}

\newcommand{\R}{\mathbb R}
\newcommand{\Z}{\mathbb Z}

\newcommand{\cA}{{\mathcal A}}
\newcommand{\cB}{{\mathcal B}}
\newcommand{\cC}{{\mathcal C}}

\newcommand{\cE}{{\mathcal E}}

\newcommand{\cW}{{\mathcal W}}

\newcommand{\bP}{\mathbf P}
\newcommand{\bQ}{\mathbf Q}

\renewcommand{\d}{\mathrm{d}}

\newcommand{\fa}{\mathfrak a}

\newcommand{\fh}{\mathfrak h}

\newcommand{\fI}{\mathfrak I}
\newcommand{\fX}{\mathfrak X}

\newcommand{\BB}{\mathsf{BB}}

\usepackage{crossreftools}

\newcommand{\bookmarkcref}[1]{\crtcrefcounter{#1}~\crtrefnumber{#1}}

\ExplSyntaxOn
\newcommand{\bookmarkCref}[1]{\text_titlecase_first:n{\crtcrefcounter{#1}}~\crtrefnumber{#1}}
\ExplSyntaxOff

\AddToHook{env/lemma/begin}{\crefalias{theorem}{lemma}}
\AddToHook{env/claim/begin}{\crefalias{theorem}{claim}}
\AddToHook{env/proposition/begin}{\crefalias{theorem}{proposition}}
\AddToHook{env/corollary/begin}{\crefalias{theorem}{corollary}}
\AddToHook{env/definition/begin}{\crefalias{theorem}{definition}}
\AddToHook{env/example/begin}{\crefalias{theorem}{example}}
\AddToHook{env/question/begin}{\crefalias{theorem}{question}}
\AddToHook{env/remark/begin}{\crefalias{theorem}{remark}}
\AddToHook{env/observation/begin}{\crefalias{theorem}{observation}}
\AddToHook{env/fact/begin}{\crefalias{theorem}{fact}}
\crefname{step}{Step}{Steps}
\crefname{part}{Part}{Parts}
\crefname{case}{Case}{Cases}
\crefname{claim}{Claim}{Claims}

\newcommand{\llb}{\llbracket}
\newcommand{\rrb}{\rrbracket}
\newcommand{\Dim}{{\sc d} }

\newcommand{\Var}{\operatorname{Var}}

\newcommand{\Bin}{\operatorname{Bin}}

\newcommand{\tv}{\text{\sc tv}}
\newcommand{\sgn}{\operatorname{sgn}}

\newcommand{\one}{\mathbbm{1}}
\renewcommand{\epsilon}{\varepsilon}

\crefformat{equation}{(#2#1#3)}
\crefrangeformat{equation}{(#3#1#4) to~(#5#2#6)}
\crefmultiformat{equation}{(#2#1#3)}{ and~(#2#1#3)}{, (#2#1#3)}{ and~(#2#1#3)}
\crefrangemultiformat{equation}{(#3#1#4) to~(#5#2#6)}{ and~(#3#1#4) to~(#5#2#6)}{, (#3#1#4) to~(#5#2#6)}{ and~(#3#1#4) to~(#5#2#6)}

\author{Amir Dembo}
\address{Amir Dembo\hfill\break
Mathematics  Department and Statistics Department\\ Stanford University\\ 
Stanford, CA 94305, USA.}
\email{adembo@stanford.edu}

\author{Eyal Lubetzky}
\address{Eyal Lubetzky\hfill\break
Courant Institute %of Mathematical Sciences
\\ New York University\\
251 Mercer Street\\ New York, NY 10012, USA.}
\email{eyal@courant.nyu.edu}

\author{Ofer Zeitouni}
\address{Ofer Zeitouni\hfill\break
Department of Mathematics\\
Weizmann Institute of Science\\
Rehovot 76100, Israel\\
and
Courant Institute\\
New York University\\
251 Mercer Street\\ New York, NY 10012, USA.}
\email{ofer.zeitouni@weizmann.ac.il}


\title{The law of $(1+1)$D SOS with an area tilt in a wedge}

\begin{document}

\begin{abstract}
Motivated by the study of the level lines of the $(2+1)$\Dim Solid-On-Solid (\abbr{sos}) model above a floor,  near the corners of the box,
we derive  the limit law of an ensemble of $K$ curves 
from a $(1+1)$\Dim \abbr{sos} model, with an area tilt, and above a wedge-shaped floor in $\{-N,\ldots,N\}$. 
We show that there exist explicit critical points $\alpha_0=1>\alpha_1>\ldots>\alpha_K>0$ such that, for each $r \geq 1$,  along the intervals $\pm(\alpha_{r} N,\alpha_{r-1} N)$,  the bottom $K+1-r$ curves,  rescaled by $(N^{2/3},N^{1/3})$, tend to the law of a Geometrically-Area-Tilted Ensemble of non-crossing Brownian paths (Brownian \abbr{gate}),
independently across those $2K$ intervals.  All other 
curves,  centered and rescaled by $(N,\sqrt{N})$,  tend to a product of $K$ suitable Brownian bridges. \end{abstract}

% {\mbox{}\vspace{-0.3in}
 \maketitle
% }
% \vspace{-0.35in}

\section{Introduction}\label{sec:intro}

The $(d+1)$\Dim Solid-On-Solid (\abbr{SOS}) model is a probability distribution over $\Z$-valued height functions $\varphi$ on $\Lambda \subset \Z^d$, defined as follows. Write $x\sim y$ for a pair of adjacent sites in $\Lambda$, and set $\varphi(x) = 0$ for all $x\notin\Lambda$ (zero boundary conditions); the probability assigned to $\varphi$ is then
\[ \frac1{Z_{\Lambda}} \exp\Big(-\beta\sum_{x\sim y} |\varphi(x)-\varphi(y)|\Big)\,,
\]
for a parameter $\beta> 0$ (the inverse-temperature) and a normalizer $Z_\Lambda = Z_{\Lambda}(\beta)$ (the partition function). Of main interest is the setting where the $\Z$-valued functions $\varphi$ are constrained to be nonnegative, referred to as a floor (or hard wall) at height $0$.

As explained in \cref{sec:lit},  the analysis of the level lines of the $(2+1)$\Dim \abbr{SOS} model 
in a box,  motivates the study of the following 
$(1+1)$\Dim \abbr{SOS}-type line ensemble,  with floor constraints in the form of a wedge and an area tilt.

Fixing $K \geq 1$,  an inverse-temperature $\beta>0$ and area tilt parameters $\fh>0$,  $\lambda>1$,  
we introduce the curves 
$\varphi_1,\ldots,\varphi_K$ such that
\begin{equation}\label{eq:1d-sos-domain-many} \varphi_k(\pm N) = 0~,~ \varphi_k(t)\geq \varphi_{k+1}(t) ~\mbox{ for all $k,t$,} \qquad \varphi_K(t) \geq |t|-N~\mbox{ for all $|t|<N$}\,,\end{equation}
assigning them the probability
\begin{equation}\label{eq:1d-sos-many-curves}\mu_{N,K}(\underline \varphi) = 
\frac1{Z_{N,K}}\exp\bigg(-
\sum_{k=1}^K \Big[ \beta \sum_{t=-N}^{N-1} \left|\varphi_k(t+1)-\varphi_k(t)\right| + 
\frac{\fh} N %\sum_{k=1}^K 
\lambda^{k-1} 
\sum_{t=-N}^N 
\varphi_k(t) \Big]
\bigg)\,,
\end{equation}
which is invariant under time reversal. 
Under this model,
each of the (ordered) $\varphi_k$ is given a geometrically-increasing area tilt $\lambda^{k-1}\fh/N$. We identify $\varphi_k$ with the piecewise linear interpolation of the height functions $\varphi_k$.

It will be convenient to set $\phi(t) = \varphi(t)-(|t|-N)$, replacing \cref{eq:1d-sos-domain-many,eq:1d-sos-many-curves} by the equivalent formulas
\begin{align}\label{eq:1d-sos-domain-many-phi} &\phi_k(\pm N) = 0~,~ \phi_k(t)\geq \phi_{k+1}(t) ~\mbox{ for all $k,t$,} \qquad \phi_K(t) \geq 0~\mbox{ for all $|t|<N$}\,, \\
& \mu_{N,K}(\underline \phi) = 
\frac1{Z_{N,K}}\exp\bigg(-
\sum_{k=1}^K \sum_{t=-N}^{N-1} \Big[ \beta \big|\phi_k(t+1)-\phi_k(t)+\sgn(t)\big| %\nonumber \\ &\qquad\qquad\qquad\qquad\qquad\qquad
+ 
\frac{\fh} N %\sum_{k=1}^K 
\lambda^{k-1} 
%\sum_{t=-N}^N 
\phi_k(t) \Big]
\bigg)\,,
\label{eq:1d-sos-many-curves-phi}
\end{align}
setting hereafter $\sgn(t)=1$ if $t\geq 0$ 
and $\sgn(t)=-1$ otherwise.

The following class of processes will appear in our main result.
\begin{definition}\label{dfn:B-gate}
We denote by 
$\mu_{\lambda,\fa,\ell}^{\mathfrak{o}}$ the probability measures on $C(\R,\R^\ell)$ 
obtained as the weak limit $T \to \infty$  of the probability measures 
$\P^{\underline{0},\underline{0}}_{\ell;-T,T}\left(\cdot\mid \fa,\lambda\right)$
from \cite[(1.10)]{CIW19b}.  The existence of such a limit follows from the tightness in 
\cite[Thm~1.3]{CIW19b} together with the monotonicity property of \cite[(3.12)]{CIW19b}.
The associated stochastic process is called a Brownian \abbr{gate}
(of $\ell$ curves and area tilt $\fa \lambda^{i-1}$ for curve~$i$).
\end{definition}
We further introduce certain relevant topologies 
on stochastic processes.
\begin{definition}\label{def:shifted-conv}
We say that stochastic processes 
$\underline{\psi}_{N}: \R \to \R^j$ converge subject to $[a_N,b_N]$-shifts,   
if for any deterministic sequence $t_N \in [a_N,b_N]$, 
the sequence $\underline{\psi}_{N}(t_N+\cdot)$ converges weakly in $C(\R,\R^j)$ uniformly on compacts.
\end{definition}

For our main result we introduce a few additional objects starting with
\begin{align}
\label{eq:rw-logmgf} \Lambda(\theta) 
%&= \log\E_0[e^{\theta \fX}] \nonumber \\
&:= \begin{cases}2\log(1-e^{-\beta})-\log(1-e^{-\beta+\theta})-\log(1-e^{-\beta-\theta}),& |\theta|<\beta,\\ \infty & |\theta|\geq \beta.\end{cases}
\end{align}
Note that $ \Lambda(\theta) = \log\E_0[e^{\theta \fX}]$
for the integer-valued random variable
%(see \cref{eq:rw-logmgf} for its explicit expression).
$\fX$ of law
\begin{equation}\label{eq:rw-law} \P_0(\fX = \ell) = \frac{1-e^{-\beta}}{1+e^{-\beta}} e^{-\beta |\ell|}\qquad(\ell\in\Z)
\,.
\end{equation}
Hereafter, 
\begin{equation}\label{eq:theta-star-def}
\theta_\star = \log\cosh(\beta)\,,
\end{equation}
which is the unique positive solution of
% (see the discussion following \cref{eq:rw-logmgf})
\begin{equation}\label{eq:theta-star}
    \Lambda'(\pm \theta_\star) = \pm 1\,.
\end{equation}
Then,  for $|u| \le 1$,  we set
\begin{align}\label{eq:phi*-def}
    m_\star(u) &= 
          1-|u| + \frac{\Lambda(\theta_\star |u| )-\Lambda(\theta_\star)}{\theta_\star} \,, \\
    \label{eq:V-def}
    v_\star(u) &= 
    \frac{1+\Lambda'\left(\theta_\star u\right)}{\theta_\star}\,, 
\end{align}
with $m_\star(u)=v_\star(u)=0$ in case $|u| \ge 1$.
Next,  for
$\fh > \theta_\star$ and $\lambda >1$,  let $\alpha_0:=1$ and 
\begin{equation}\label{eq:alpha-k-def}
\alpha_k := \frac{\theta_\star}{\fh \lambda^{k-1}}\,,\quad
\fa_k := \frac{\theta_\star}{\alpha_k}\sqrt{\Lambda''(\theta_\star)}\,, \quad
\phi_k^\star(t) := \alpha_k m_\star\big(\tfrac{t}{\alpha_k}\big)\,,  \quad 
1 \le k \le K\,.
\end{equation}

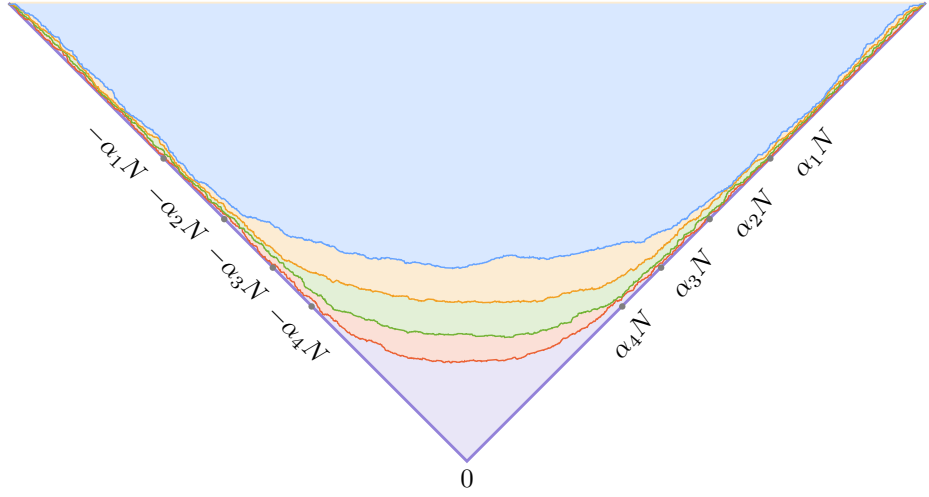
\begin{figure}
%\vspace{-0.22in}
    \begin{tikzpicture}
    \node (fig1) at (0,0) {    \includegraphics[width=.87\textwidth]{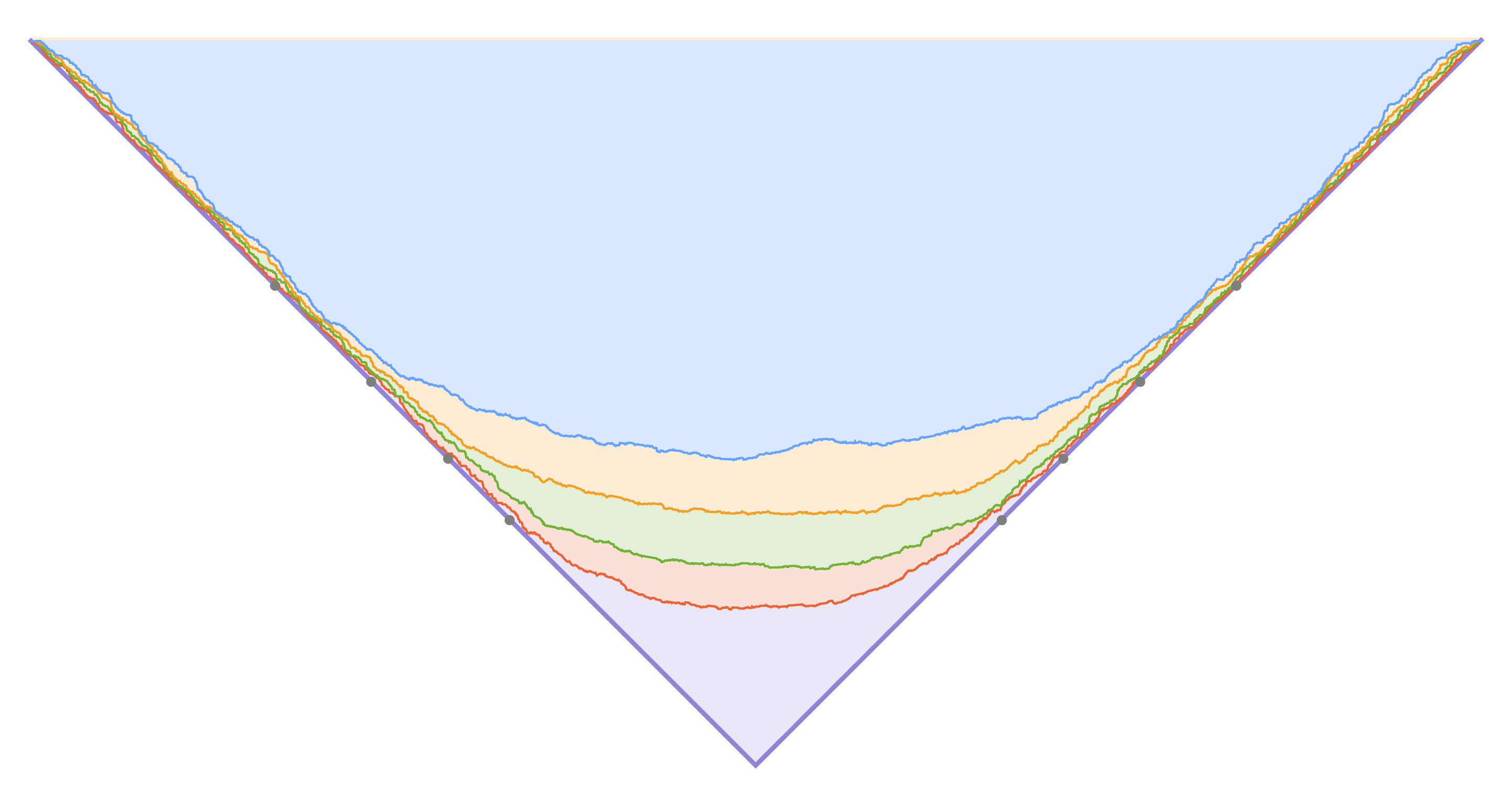}};
    \node[font=\small, rotate=-45] (am1) at (-4.65,1.05) {$-\alpha_1 N$};
    \node[font=\small, rotate=50] (ap1) at (4.6,1.05) {$\alpha_1 N$};
    \node[font=\small, rotate=-45] (am2) at (-3.85,0.25) {$-\alpha_2 N$};
    \node[font=\small, rotate=50] (ap2) at (3.8,0.25) {$\alpha_2 N$};
    \node[font=\small, rotate=-45] (am3) at (-3.05,-0.55) {$-\alpha_3 N$};
    \node[font=\small, rotate=50] (ap3) at (3.,-0.55) {$\alpha_3 N$};
    \node[font=\small, rotate=-45] (am4) at (-2.25,-1.35) {$-\alpha_4 N$};
    \node[font=\small, rotate=50] (ap4) at (2.2,-1.35) {$\alpha_4 N$};
    \node[font=\small] (zero) at (0,-3.25) {$0$};
    \end{tikzpicture}
    \vspace{-0.15in}
    \caption{The $(1+1)$\Dim \abbr{SOS} from \cref{eq:1d-sos-many-curves-phi} with $N=1000$ and $K=4$ curves. The $\alpha_k$'s, per \cref{eq:alpha-k-def}, mark the transition from flat to curved scaling limits.}
   % \vspace{0.3in}
    \label{fig:sos-sim}
\end{figure}

\begin{theorem}
    \label{thm:1}
    Fix $\beta>0$, $\lambda>1$, $\fh > \theta_\star$ and $K\geq 1$. 
    Consider the line ensemble $\{\phi_k\}_{k=1}^K$ on $[-N,N]$ with a geometric area tilt,  as in  \cref{eq:1d-sos-domain-many-phi,eq:1d-sos-many-curves-phi}.  Then,
    \begin{enumerate}[(a), leftmargin=*]
\item \label[part]{part:a} % (a) 
As $N \to \infty$,  the $\R^K$-valued process of coordinates 
 \[ 
    \widehat\phi_k(t) := \frac{\phi_k(N t) - N\phi_k^\star(t)}{\sqrt {\alpha_k \,  N}}
 \]
(linearly interpolated),  converges weakly on $C([-1,1],\R^K)$ to rescaled,  independent Brownian bridges 
$(\BB_1, \ldots,\BB_K)$,  where for \abbr{iid} standard Brownian motions $W_k$,
\begin{equation} \label{def:un-X}
\BB_k (t) := \left\{
\begin{array}{ll}
W_k(v_\star(t/\alpha_k))-\frac{v_\star(t/\alpha_k)}{v_\star(1)}  W_k(v_\star(1))
\,,  & |t| \le \alpha_k \,,\\
0\,, & \alpha_k \le |t| \le 1 \,.
\end{array} \right.
\end{equation}
\item \label[part]{part:b} %(b)  
Fix $\epsilon \in (0,\frac{1}{12})$ and consider the intervals
\[ 
I_N^{(-r)}:=-I_N^{(r)},  \qquad 
I_N^{(r)} := [\alpha_{r} N^{1/3} + 2 N^\epsilon,\alpha_{r-1} N^{1/3} - 2 N^\epsilon] \,,  \qquad 1 \le r \le K \,.
\]
Then,  for each $r=-K,\ldots,-1,1,\ldots, K$,  the $\R^{K+1-|r|}$-valued processes
\[ 
\underline{\psi}^{(r)}_N (u) := 
\frac{1}{\sqrt{\Lambda''(\theta_\star)}N^{1/3}} \big(\phi_{|r|} (N^{2/3} u),\ldots,\phi_K(N^{2/3} u) \big) 
 \]
converge 
subject to $I_N^{(r)}$-shifts
to the Brownian \abbr{gate} $\underline{X}^{(r)}$ 
of law 
$\mu^{\mathfrak o}_{\lambda,\fa_r,K-|r|+1}$.
Further,  the convergence
holds jointly over $r$,  with mutually independent limits.
\end{enumerate}
\end{theorem}

We note that
the case $K=1$ of a single curve (which is motivated by the Ising model) was recently studied in \cite{OSV24}; our work answers in the affirmative one of their conjectures: see \cref{rem-Senya} below.

\begin{remark}\label{rem:ofer1}  \cref{part:a}
of our main theorem says that the $k$-th curve from the top
separates from the curves below it during $[-\alpha_k N,\alpha_k N]$,  and upon centering around
$N \phi^\star_k(\cdot)$ those separated curves jointly possess at $(N,\sqrt{N})$-diffusive scaling 
the limit law of the independent Brownian bridges $(\BB_1,\ldots,\BB_K)$.  
As seen in \cref{fig:sos-sim,fig:scaling-limit},  while all $K$ curves are separated during 
$[-\alpha_K N, \alpha_K N]$,  fixing $1 \le r \le K$,  the $(K+1-r)$ bottom 
curves have not separated during $[-\alpha_{r-1}N,-\alpha_r N] \cup [\alpha_r N, \alpha_{r-1} N]$. 
By \cref{part:b},  at the slower
$(N^{2/3},N^{1/3} \sqrt{\Lambda''(\theta_\star)})$-diffusive scaling,  these non-separated curves
converge in law as $N \to \infty$ to the corresponding 
Brownian \abbr{gate}.  While the rescaled time intervals are of
$\Theta(N^{1/3})$ width,  our convergence is local,  over $O(1)$ limit intervals,
arbitrarily centered (up to $\Omega(N^\epsilon)$ distance from each transition point), 
with mutually independent instances of the limiting Brownian \abbr{gate}.
\end{remark}

\begin{figure}
{\definecolor{wedgepurple}{RGB}{140,30,140}
\definecolor{wedgegreen}{RGB}{0,140,0}
\definecolor{wedgeblue}{RGB}{30,30,230}
\definecolor{wedgeorange}{RGB}{240,140,30}
\vspace{-0.1in}
\begin{tikzpicture}
  \begin{scope}
    \node[draw, thin, gray, inner sep=2pt] (fig1) at (0,0) {    \includegraphics[width=.35\textwidth]{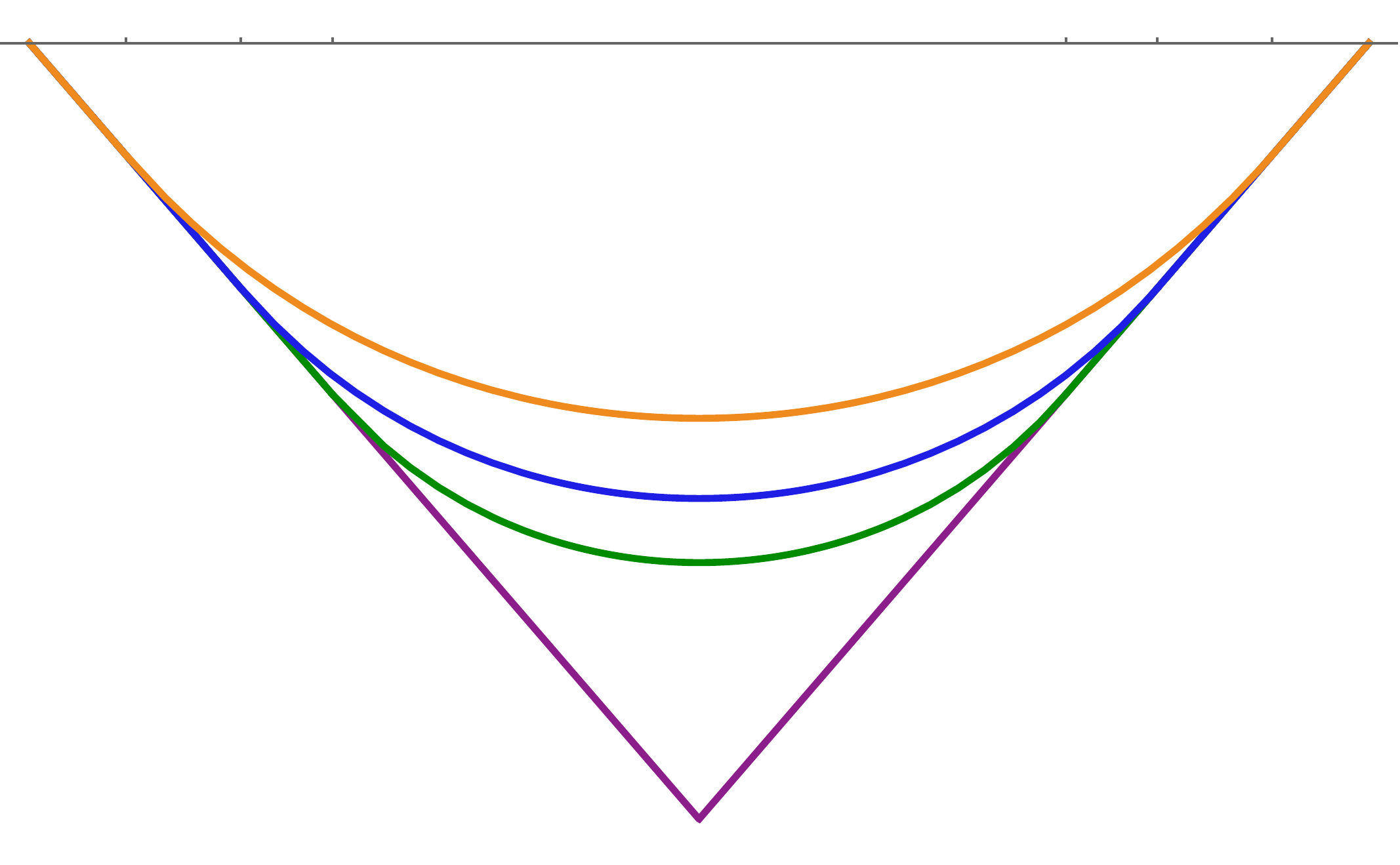}};

    \foreach \x/\lab/\col in {-2.08/{-\alpha_1 N}/wedgeorange, -1.66/{-\alpha_2 N}/wedgeblue,
                             -1.33/{-\alpha_3 N}/wedgegreen,  1.33/{\alpha_3 N}/wedgegreen,
                              1.66/{\alpha_2 N}/wedgeblue,    2.08/{\alpha_1 N}/wedgeorange}
    \draw (\x,1.36) -- (\x,1.46) node[xshift=1pt,yshift=8pt,font=\tiny,\col,rotate=45] {$\lab$};

    \draw[black,dotted,thick] (-1.91,0.16) rectangle ++(0.5,0.5);

    \node[below,wedgegreen,font=\tiny] at (0,-0.4) {$\varphi_1$};
    \node[below,wedgeblue,font=\tiny]   at (0,0.1) {$\varphi_2$};
    \node[below,wedgeorange,font=\tiny]  at (0,0.5) {$\varphi_3$};
  
  \end{scope}

    \begin{scope}[yshift=-95pt]
        \node[draw, thin, gray, inner sep=2pt] (fig2) at (0,0) {    \includegraphics[width=.35\textwidth]{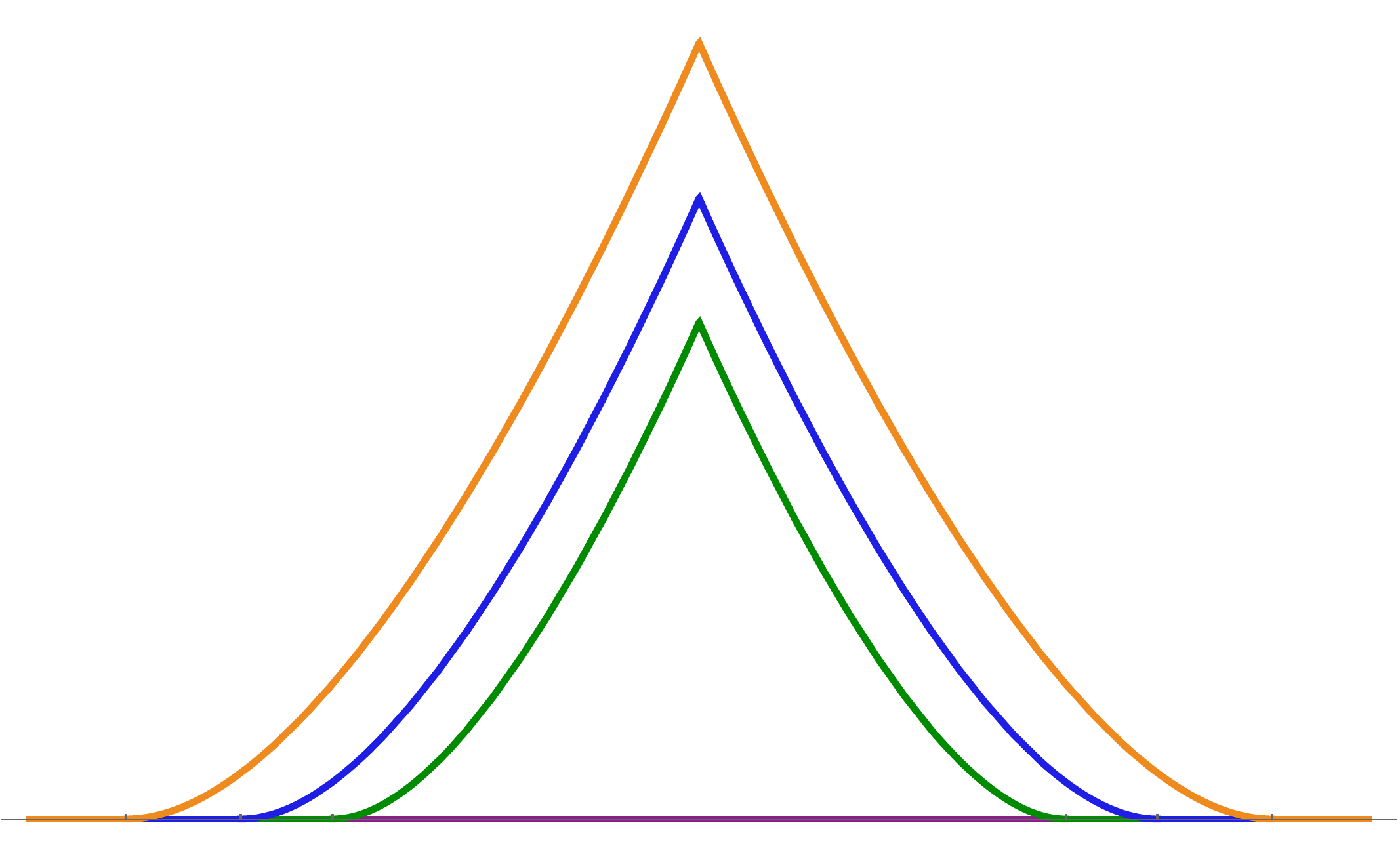}};

    \foreach \x/\lab/\col in {-2.08/{-\alpha_1 N}/wedgeorange, -1.66/{-\alpha_2 N}/wedgeblue,
                             -1.33/{-\alpha_3 N}/wedgegreen,  1.33/{\alpha_3 N}/wedgegreen,
                              1.66/{\alpha_2 N}/wedgeblue,    2.08/{\alpha_1 N}/wedgeorange}
    \draw (\x,-1.36) -- (\x,-1.46) node[left=-3pt,font=\tiny,\col,rotate=45] {$\lab$};

    \draw[black,dotted,thick] (-1.91,-1.48) rectangle ++(0.5,0.5);

    \node[below,wedgegreen,font=\tiny] at (0,0.25) {$\phi_1$};
    \node[below,wedgeblue,font=\tiny]   at (0,0.8) {$\phi_2$};
    \node[below,wedgeorange,font=\tiny]  at (0,1.3) {$\phi_3$};
    \end{scope}

      \begin{scope}[xshift=180pt,yshift=-45pt]
  \node (fig3) at (0,0) {    \includegraphics[width=.5\textwidth]{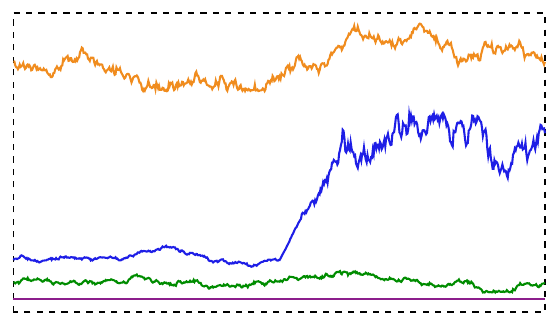}};

      \node[below,wedgeblue,font=\small]   at (0,-0.2) {$\widehat\phi_2$};

    \node[text=gray,font=\small] (bblabel) at (1.75,-2.5) {$(N,\sqrt{N})$-rescaled $\BB_2$};

    \node[text=gray,font=\small] (bblabel) at (-1.3,-3.1) {$(N^{2/3},N^{1/3})$-rescaled ${\underline X}^{(2)}$};

    \draw[very thin, gray,-{Stealth[scale=1.5]}] (1.75,-2.3)--(1.75,0.1);

    \draw[very thin, gray,-{Stealth[scale=1.5]}] (-1.3,-2.9)--(-1.3,-1.7);

    \end{scope}  
\end{tikzpicture}
}
\vspace{-0.1in}
\caption{Illustration of the transformation from $\{\varphi_k\}$ to $\{\phi_k\}$ and the limit law of \cref{thm:1} for $\{\widehat\phi_k\}$.}
\label{fig:scaling-limit}
\end{figure}
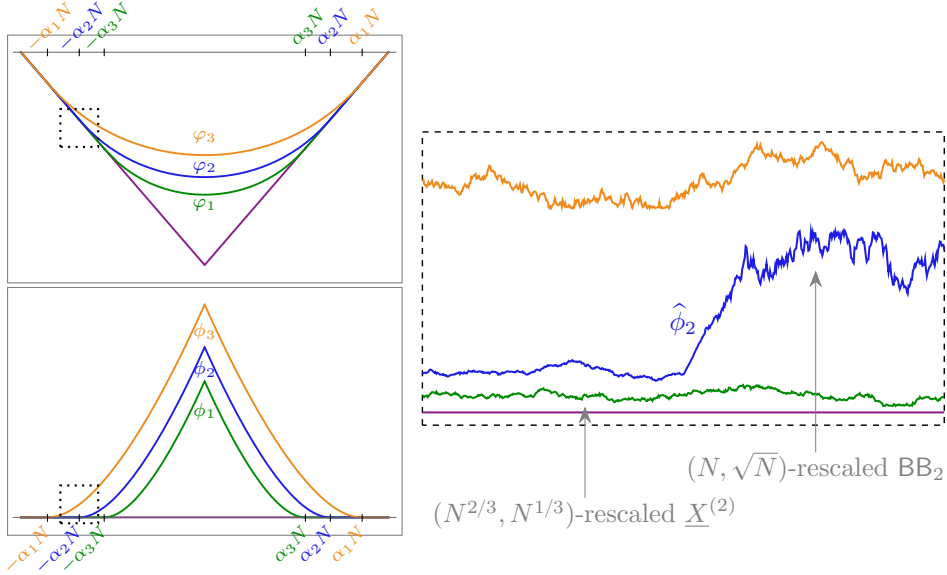

\begin{remark}
Our assumption that $\lambda>1$ is an essential part of the model definition (whereby the area tilt increases geometrically for lower level lines), whereas the assumption $\mathfrak{h}>\theta_\star$ is  just to simplify the statement: without it,  the $\phi_i$'s with $\mathfrak{h} \lambda^{i-1} \leq \theta_\star$ do not have a flat portion and converge to independent Brownian bridges on $[-1,1]$.
\end{remark}
\begin{remark} Our proof of  \cref{thm:1} actually yields the joint convergence of the processes in \cref{part:a,part:b} to independent limits. In order to avoid cumbersome notation and definitions, we chose to state the theorem without this extension.
\end{remark}
\begin{remark}
We do not address  the interesting case where the number of curves $K$ tends to infinity with $N$. (The particular case $K=\Theta(\log N)$ is pertinent due to its link with the  $(2+1)$\Dim  \abbr{sos} model, as discussed in the next section).
Most of our arguments (perhaps except for  \cref{lem:N-third-tight})  carry over as soon as $K$ grows sub-polynomially, but many  adaptations are needed, and the 
notation and even statements become rather cumbersome. 
\end{remark}

\subsection{Motivation and related works}\label{sec:lit}
The \abbr{SOS} model was introduced in the early 1950s, in $(1+1)$\Dim by Temperley~\cite{Temperley52} and in $(2+1)$\Dim
by Burton, Cabrera and Frank~\cite{BCF51}. It has been extensively studied as a model for the low temperature evolution of crystals and the plus/minus interfaces in the Ising model. The roughening transition, which remains a tantalizing open problem awaiting rigorous confirmation in 3\Dim Ising, was verified for the (much more tractable) $(2+1)$\Dim \abbr{sos}. There, it was shown that the height at the origin $o$ in a box of side-length $N$ moves from having $\Var(\varphi(o))=O(1)$ at low temperature (\cite{BrandenbergerWayne82}) to $\Var(\varphi(o))\asymp \log N$ at high temperature (the celebrated works \cite{FrohlichSpencer81a,FrohlichSpencer81b}).

The low temperature picture becomes considerably richer in the
presence of a floor (restricting $\varphi$ to $\Z_+$).  For the $(2+1)$\Dim \abbr{sos} in a box,  the pioneering work of Bricmont, El-Mellouki and Fr\"ohlich~\cite{BEF86} showed that this creates \emph{entropic repulsion}: for large $\beta$, the height of the surface in the bulk (which, in the absence of a floor, is concentrated about $0$ with exponential tails) is propelled to order $\log N$. 
By now much more is known about the shape of the surface in this scenario. The works
\cite{CLMST14,CLMST16} established that the level lines of the surface typically form a sequence of either $\lfloor \frac1{4\beta} \log N\rfloor$ or $\lfloor \frac1{4\beta} \log N\rfloor-1$ nested macroscopic loops. The innermost loops, $\mathfrak{L}_1,\mathfrak{L}_2,\ldots$, each have a macroscopic scaling limit that reveals a Wulff shape near the $4$ corners of the box,  while coinciding with the side of the box near the center-sides.  Along these ``flat'' portions, 
the level lines are conjectured to have the scaling limit of a sequence of Brownian polymers,  each tilted by the area below it via a factor of about $\lambda^k / N$,  where $k$ is the index of the level line, conditioned to not cross one another.
A single such curve (in \abbr{sos} there are $\asymp \log N$ interacting curves) has $N^{1/3}$ fluctuations and the limit law of a Ferrari--Spohn~\cite{FerrariSpohn05} diffusion.  In \abbr{sos} near the center-side, the fluctuations are at most $N^{1/3+o(1)}$ (\cite{CLMST16}) and of order at least $N^{1/3}$ (\cite{CKL24}), but the limiting law remains out of reach. See \cite{CIW19a,CIW19b,DLZ24,CG23,Serio23,HKS25,BCG25} for works on  $(1+1)$\Dim models approximating the $(2+1)$\Dim level lines along the flat portion of the limit: a Geometrically-Area-Tilted Ensemble of Brownian polymers---a Brownian \abbr{gate}, and its discrete counterpart, a \abbr{rw} \abbr{gate}.

Near the corners of the box, where the scaling limit is nontrivial, the level lines are conjectured to behave as a Brownian motion once centered about their mean: see, e.g., the discussion following Eq.~(3.6) in~\cite{SchonmannShlosman95}, and the following from \cite{CLMST16}: ``{\em it seems natural to conjecture that normal fluctuations appear along the curved portions, where the limiting shapes corresponding to distinct levels are macroscopically separated.}''
However, the only known upper bound on these fluctuations is $N^{1-\epsilon}$ and
no nontrivial lower bound is known. Of particular difficulty is understanding the behavior of the level lines as their scaling limit transitions from flat (along the boundary) to curved (near the corners), and the effect this has on the law in the bulk of the curved portion.

Just as the works on ensembles of non-crossing $(1+1)$\Dim \abbr{sos} curves with geometric area tilts mimicked the $(2+1)$\Dim level lines near the flat portion of their limit, placing this ensemble above a wedge shaped wall mimics the $(2+1)$\Dim behavior in a quadrant of the box, including around the transition points from flat to curve limits (see \cref{fig:sos-sim}).

To explain our specific law \cref{eq:1d-sos-many-curves} 
for polymers in a wedge,  let us elaborate on the known behavior of the level line loops in the $(2+1)$\Dim \abbr{sos}. Consider values of $N$ such that,  as per \cite{CLMST16}, the \abbr{sos} surface in a $2N\times 2N$ box typically forms a plateau at height $\lfloor\frac1{4\beta}\log N\rfloor$ (as opposed to $\lfloor \frac1{4\beta} \log N\rfloor-1$). To zoom in on (and simplify) a quadrant of this surface, let us consider an $N\times N$ box with boundary conditions that are $\lfloor \frac1{4\beta} \log N\rfloor$ on the top and right sides, while they are  $\lfloor \frac1{4\beta} \log N\rfloor - K$ for some fixed $K\geq 1$ on the bottom and left sides. This forces $K$ level lines $\gamma_1,\ldots,\gamma_{K}$ connecting the upper-left and bottom-right corners (an $h$ level line is a path consisting of edges dual to $x\sim y$ where the \abbr{sos} configuration is at least $h$ at $x$ and less than $h$ at~$y$; locations where $4$ such edges touch one vertex are resolved via a global splitting rule, making the resulting path simple; in the above, $\gamma_k$ corresponds to $h=\lfloor\frac1{4\beta}\log N\rfloor-k+1$).
One of the key ingredients in the analysis of \cite{CLMST16} was that, in such a scenario,
\[ \P(\gamma_1,\ldots,\gamma_{K}) \propto \exp\bigg(-\sum_{k=1}^K\Big[\beta|\gamma_k| + \widehat{\pi}_\infty\left(\varphi(o) = \lfloor\tfrac1{4\beta}\log N\rfloor-k+1\right)|V_{\gamma_k}|+\mathrm{Errs}\Big]\bigg)\,,\]
where $|\gamma_k|$ is the level-line length (number of edges), $\widehat{\pi}_\infty$ is the $\infty$-volume measure on \abbr{sos} configurations $\varphi$ at inverse-temperature $\beta$, 
and $V_{\gamma_k}$ is the area sandwiched between $\gamma_k$ and the left and bottom sides. (See \cite[Eq.~(1.6) and Prop.~2.16]{CLMST16} for the high-level justification of this via cluster expansion and formal corresponding estimate; see also \cite[\S4]{ChenLubetzky26} and in particular \cite[Prop.~4.3, Cor.~3]{ChenLubetzky26} where such estimates were established for the $|\nabla\varphi|^p$ models for all $p\geq 1$; \abbr{sos} is $p=1$; 
$\Z$\abbr{gff} (Discrete Gaussian) is $p=2$). Note that the entropic repulsion phenomenon is visible in this formula:
\begin{enumerate}[(a)]
\item Energy: Every edge of $\gamma_k$ represents one unit of height discrepancy (the $\gamma_k$ paths are non-crossing, but they may overlap, which occurs whenever the gradient along them is larger than $1$), hence the energetic term $\exp(-\beta \sum_k |\gamma_k|)$. 
\item Entropy: If $h = \lfloor\frac1{4\beta}\log N\rfloor - k + 1$, then $\gamma_k$ delimits the rigid $h$ and $h-1$ phases, and the sites in $V_{\gamma_k}$ are in the $h-1$ phase. There, the surface must (mostly) forbid downward deviations of $h$ to remain nonnegative; the associated  entropic term is then $\approx \prod_{x\in V_{\gamma_k}}(1-\widehat\pi_\infty(\varphi(x)  = -h)) \approx \exp[-\widehat\pi_\infty(\varphi(o) = -h)|V_{\gamma_k}|]$. 
\end{enumerate}
In the \abbr{sos} model, $\widehat\pi_\infty(\varphi(o)=h) = (c+o(1))e^{-4\beta h}$ %for $c=1\pm \epsilon_\beta$ 
(cf.~\cite[Lem.~2.4]{CLMST16}), and we arrive at
\[ \P(\gamma_1,\ldots,\gamma_{K}) \propto \exp\bigg(-\sum_{k=1}^K\Big[\beta|\gamma_k| + \frac{c}{N} e^{4\beta (k-1)} |V_{\gamma_k}|+\mathrm{Errs}\Big]\bigg)\,.\]
(Of~course, the term $\mathrm{Errs}$ above hides highly nontrivial interactions between $\gamma_i,\gamma_j$ (beyond non-crossing), between $\gamma_i$ and itself, and between $\gamma_i$ and the boundary of the box. These interactions are the main obstacle 
in the analysis of these \abbr{sos} level lines.) 

It is well-known that if $\P(\gamma)\propto \exp(-\beta|\gamma|+\mathrm{Errs})$, i.e., with no area tilt, and $\beta$ is large enough, then under cluster-expansion-type assumptions on the $\mathrm{Errs}$ (satisfied by the \abbr{sos} estimates), the scaling limit of $\gamma$ would be a Brownian bridge (see, e.g., \cite{Higuchi79}, where this was applied to recover the limit of the 2\Dim Ising interface at low temperature). 

The extra area term $|V_{\gamma_k}|$ appearing above motivated a line of studies on random walks and Brownian motion polymers penalized by the area below them. For a single random walk with this area tilt, the limit is known~\cite{ISV15} to be a Ferrari--Spohn diffusion. Recently, using Ornstein--Zernike machinery, \cite{IOSV21} resolved the longstanding problem of showing that the low temperature 2\Dim Ising interface with critical prewetting has a Ferrari--Spohn limit; the law of this interface $\gamma$ is roughly $\exp[-(\beta|\gamma|+\frac{c}N|V_{\gamma}|+\mathrm{Errs})]$, and indeed the Ornstein--Zernike framework allowed the authors to reduce the problem to an area-tilted random walk, whose analysis thereafter yielded the sought limit law.
For a finite collection of such walks with the same area tilt, the limit is known~\cite{IVW18} to be a set of non-crossing  Ferrari--Spohn diffusions.  (In the $|\nabla\varphi|^p$ models, it is known~\cite{ChenLubetzky25} that, for any $p>1$
%and $m$ fixed,  
the scaling limit of the top $m$ level lines is that of $m$ \abbr{iid} Ferrari--Spohn diffusions; however,  the \abbr{sos},  where $p=1$,  should behave differently).

Studying Brownian motion/random walks where $\gamma_k$ is tilted by the area below it, with a prefactor that changes with $k$, is already nontrivial: the dependency on $k$ destroys exchangeability, and already deriving the correct scale is nontrivial. 
For geometrically increasing area tilts, as in the formula for $\P(\gamma_1,\ldots,\gamma_k)$, the works~\cite{CIW19a,CIW19b,DLZ24,CG23,BCG25} studied the continuous line ensemble: Brownian polymers where the $k$-th one, $X_t^{(k)}$, is tilted by $\exp(-\mathfrak{a} \lambda^{k-1} \int X_s^{(k)}\d s)$. (The factor $1/N$ scales out when moving to Brownian motion.) 
The stationary limit law $\mu_{\lambda,\mathfrak{a},K}^{\mathfrak{o}}$
 of the Brownian \abbr{gate} with $K$ lines,  is the same when the boundary conditions on the polymers are either $0$ or free.  Some of these works (in particular, \cite{CIW19b,CG23}) allowed for taking the limit $K\to\infty$, and a spectral gap in that situation was recently proved in \cite{GW26}. The works \cite{Serio23,HKS25} studied the delicate discrete Random Walk Geometrically-Area-Tilted Ensemble
(\abbr{rw} \abbr{gate}), 
 which has the same limit law.

All these works considered the ensemble of area tilted polymers/random walks above a flat boundary,   where the fluctuations are of order $N^{1/3}$.  Our work sheds light on the \abbr{sos} corner behavior, and the conjectured transition of its level-line fluctuations from order $ N^{1/3}$ to order $\sqrt{N}$.  The ensemble of area-tilted $(1+1)$\Dim \abbr{sos} lines over wedge,  as in \cref{eq:1d-sos-domain-many,eq:1d-sos-many-curves},
approximates the $(2+1)$\Dim \abbr{sos} level lines in the same vein of the Brownian polymers/random walk approximation, and should serve as a stepping stone to analyzing the $(2+1)$\Dim \abbr{sos} (just as the random walk results were used in the study of the 2\Dim Ising and $(2+1)$\Dim \abbr{sos}/$\Z$\abbr{GFF}/$|\nabla\varphi|^p$ models via Ornstein--Zernike tools). 

\subsection{Proof outline}
The main challenge in establishing \cref{thm:1} stems from 
%having 
the normalizing partition function $Z_{N,K}$ in \cref{eq:1d-sos-many-curves-phi}.  With multiple curves, 
%(and the difficulty arises already when there are two),  
as explained in \cref{rem:ofer1},
the system behaves differently in different segments of the interval $[-N,N]$: one by one,  curves disengage from the flat portion,  where the limiting law of $\ell$ curves is the aforementioned $\mu_{\lambda,\fa,\ell}^{\mathfrak{o}}$ for suitable $\lambda,\fa$, and become Brownian bridges (with $\sqrt N$ fluctuations) in the remaining ``middle'' interval.  One has only an expression for the entire system, involving $Z_{N,K}$,  and the main challenge is to extract from it a tractable expression for the area-tilted subsystem on one of these segments.
Equivalently,  the law of the heights of the processes at the endpoints of such a segment is nontrivial,  and one cannot easily disentangle the system and arrive at a tractable collection of Markov processes.  

To handle this,  we first perform in \cref{sec:ch-meas} a careful change of measure that trades the 
area tilt of $\phi_i$ in $[-\alpha_i N,\alpha_i N]$ with curves $\widetilde{\phi}_i$ of 
drift roughly $c_i (j/N)$ at $\pm (\alpha_i N - j)$ with $c_i>0$,  drawn from $\Q_\fh (\cdot)$ of \cref{def:Qh} 
prior to being ordered (cf.~\cref{lem:tilting}).  \cref{sec:prelim} introduces two key
ingredients for our proof: a stochastic ordering of line ensemble laws and,  
following \cite{ISV15}, a bound on the height of a single area-tilted \abbr{rw} (adapted to the free right-end boundary we have here).  Combining these with estimates from \cite[Sec.~5]{HKS25},
\cref{prop:N-third-almost-tight} shows that under $\Q_\fh$,  the area-tilted parts left of $-\alpha_i N$
(called $\Phi^-$),   and on the right of $\alpha_i N$ (called $\Phi^+$), 
%(on the left side and on the right side),
are \abbr{whp} of height at most $ N^{1/3+\epsilon}$,  while \cref{prop:v1-pm}
shows that,  as expected,  $\widetilde{\phi}_i(-\alpha_i N + \Delta)$ concentrates 
at height $C \frac{\Delta^2}{N}$ for $N \gg \Delta \gg N^{2/3}$.  This allows us to   
effectively omit the constraints 
$\widetilde{\phi}_i(\cdot) \ge \widetilde{\phi}_{i+1}(\cdot)$ throughout $(-\alpha_i N + \Delta,\alpha_i N -\Delta)$,
resulting in a product law
over the \abbr{rw} \abbr{gate} pieces $\Phi^-$,  $\Phi^+$, 
and the separated middle part $\Phi^o$ (which has no area tilt),
subject to those three parts having matching boundary values (see 
\cref{lem:non-crossing-BB,lem:Q-A}).  Thanks to a local \abbr{clt} for the boundary of 
$\Phi^o$ in the absence of $\Phi^{\mp}$,  said boundary values remain $o(\sqrt{N})$ 
under our joint law (see \cref{lem:v1-pm}),  yielding in \cref{subsec:X0} the convergence
of \cref{part:a}, to the collection of independent Brownian bridges.
\cref{lem:v1-pm} further decouples the area-tilted parts $\Phi^-$ and $\Phi^+$,  whereby 
the tightness in \cref{lem:N-third-tight} allows us to complete the proof of \cref{thm:1} via
conditional convergence properties from \cite{Serio23}.

\subsection*{Acknowledgments} We thank Christian Serio for helpful discussions concerning \cite{Serio23} and \cite{HKS25}. This work was supported by US-Israel BSF grant \#2024020.

\section{Change of measure and decomposition}\label{sec:ch-meas}

Recall that $\Lambda(\cdot)$ of \eqref{eq:rw-logmgf} is the log-\abbr{mgf} 
of the random variable $\fX$ of \cref{eq:rw-law}. 
% In particular,  
%\begin{align}
%\label{eq:rw-logmgf} \Lambda(\theta) &= \log\E_0[e^{\theta \fX}] \nonumber \\
%&= \begin{cases}2\log(1-e^{-\beta})-\log(1-e^{-\beta+\theta})-\log(1-e^{-\beta-%\theta}),& |\theta|<\beta,\\ \infty & |\theta|\geq \beta.\end{cases}
%\end{align} 
Observe that 
$\Lambda(\cdot)$ is even (by the symmetry of $\P_0(\fX\in\cdot)$),  $\Lambda'(\cdot)$ is increasing in $(-\beta,\beta)$ (by the strict convexity of $\Lambda(\cdot)$), and the odd function $\theta\mapsto \Lambda'(\theta)$ is $C^\infty$ in the interval $(-\beta,\beta)$, going to $\pm\infty$ as $\theta\to\pm \beta$; in particular, there is a unique solution to $\Lambda'(\pm x) = \pm 1$, so $\theta_\star$ from \cref{eq:theta-star} is well-defined (as are the values $1>\alpha_1>\alpha_2>\cdots>\alpha_K>0$ from \cref{eq:alpha-k-def}). One can  verify from \cref{eq:rw-logmgf} that $\Lambda'(\theta)=\sinh(\theta)/(\cosh(\beta)-\cosh(\theta))$ for $|\theta|<\beta$,  hence 
$\theta_\star = \log\cosh(\beta)$ as stated in \cref{eq:theta-star-def}.

Next, define the random variable $\tilde\fX$ to be the tilted version of the law of $\fX$ given by
\begin{equation}
    \label{eq:P-theta}
    \P_\theta(\tilde\fX=\ell) =  \frac{1-e^{-\beta}}{1+e^{-\beta}} 
     e^{-\beta|\ell|+\theta \ell - \Lambda(\theta)}\qquad (\ell\in\Z)\,,
\end{equation}
noting that 
\begin{equation}
    \label{eq:E-theta}
    \E_\theta(\tilde\fX) = \Lambda'(\theta) \qquad \mbox{and}\qquad
    \Var_\theta(\tilde\fX) = \Lambda''(\theta)\,.
\end{equation}
In the sequel we use, for $-N\leq j \leq N$, the tilting
\begin{equation}\label{eq:theta_k-def}
\theta_k(j) =  \big( (1 \wedge \tfrac{j+0.5}{N \alpha_k}) 
\vee -1 \big) \theta_\star \,,
%\label{eq:theta_k-def}
\end{equation}
which coincides with $\pm \theta_\star$ outside the interval $[a_k^-,a_k^+)$, where
\begin{equation}\label{def:akpm}
a_k := \alpha_k N + \tfrac{1}{2} , \qquad 
a_k^\pm := \pm \lfloor a_k \rfloor \,, \qquad \{a_k\} = a_k - a_k^+ \in [0,1) \,.
\end{equation}
It is easy to check that by \cref{eq:alpha-k-def,eq:phi*-def},
for all $1 \le k \le K$ and $t \in [-1,1]$, 
\begin{equation}\label{eq:phi*-def-as-int}
    \phi_k^\star(t) = \int_{-1}^t \bar\phi_k^\star(s)\,\d s \ge 0, 
    \quad\mbox{where}\quad \bar\phi_k^\star(s) = \begin{cases}
        \Lambda'\left(\frac{\theta_\star}{\alpha_k} s \right) - \sgn(s), 
        & |s|\leq \alpha_k, \\
        0, & |s|\geq \alpha_k
        \end{cases}\,,
\end{equation}
and likewise, from \cref{eq:theta-star,eq:V-def} it follows that for any $t \in [-\alpha_k,\alpha_k]$,
\begin{equation}
    \label{eq:V-def-as-int}
    v_k(t) := \alpha_k v_\star\Big(\frac{t}{\alpha_k}\Big) =
    \int_{-\alpha_k}^t \Lambda''\Big(\frac{\theta_\star}{\alpha_k} s\Big)\,\d s\,.
\end{equation}
In view of \cref{eq:E-theta,eq:theta_k-def}, we see that 
for some $V=V(\beta)<\infty$ and all $j,k,N$,
\begin{equation}\label{eq:var-bd}
0 < \inf_{|\theta|\leq\theta_\star} \Lambda''(\theta) \le 
\Var_{\theta_k(j)}(\tilde\fX) \le 
\sup_{|\theta|\leq\theta_\star} \Lambda''(\theta) \le V \,. 
\end{equation}
Moreover, for some $C=C(\beta)<\infty$ and all $k$, $N$
\begin{align}\label{eq:phi-star-equiv}
\bigg|N \phi_k^\star(t) - \sum_{j=-N}^{\lfloor Nt\rfloor-1} \Big[\E_{\theta_k(j)}(\tilde\fX) - \sgn(j)\Big] \bigg| & \le C \,,
\qquad \forall t \in [-1,1]\,, \\
\bigg|N v_k(t) - \sum_{j=a_k^-}^{\lfloor Nt\rfloor-1} \Var_{\theta_k(j)}(\tilde\fX) \bigg| & 
\le C \,, \qquad \forall t \in [-\alpha_k,\alpha_k] \,.
\label{eq:vk-equiv}
\end{align}
(Indeed,  this follows from \cref{eq:E-theta,eq:theta_k-def,def:akpm,eq:phi*-def-as-int,eq:V-def-as-int}, comparing the Riemann sum to the integral.)

In light of \cref{eq:V-def-as-int,eq:phi-star-equiv,eq:vk-equiv}, it suffices to prove the analog of \cref{thm:1} \abbr{wrt} the discrete counterparts of $\phi_k^\star$ and $v_k$, involving $\frac1N \sum \E_{\theta_k(j)}[\tilde\fX]$ and $\frac1N\sum \Var_{\theta_k(j)}(\tilde\fX)$, resp.
To this end, we proceed with a more convenient representation 
of the law $\mu_{N,K}$ of interest, as the conditioning on non-crossing, 
for a weighted product measure $\Q_{\fh}$.

\begin{definition}\label{def:Qh}
For independent random variables $\tilde \fX_k(j) \sim \P_{\theta_k(j)}$,  we denote by $\Q_\star(\cdot)$
the joint law of $\underline{\widetilde{\phi}}(\cdot):=(\widetilde{\phi}_1(\cdot),\ldots,\widetilde{\phi}_K(\cdot))$,
where 
\begin{equation}\label{eq:tilde-phi-def}
    \widetilde \phi_k(i) = \sum_{j=-N}^{i-1} \big[ \tilde \fX_k(j) - \sgn(j) \big],    \qquad i \in [-N,N] \,,  \;\;
    1 \le k \le K \,,  
\end{equation} 
using $\E_\star(\cdot)$ for the associated expectation. 
%(in particular $\widetilde{\phi}_k(-N)=0$).  starting at  
Let $\Q_\fh$ be the 
probability measure whose density \abbr{wrt} $\Q_\star$ is 
proportional to $e^{-\fh \cW^-} e^{-\fh \cW^+}$, 
where
\begin{align}\label{eq:tilted}
\cW^- &:= \sum_{k=1}^K \frac{\lambda^{k-1}}{N} 
\sum_{t=-N}^{a_k^-} \big(\widetilde \phi_k(t) \vee 0\big)
- \sum_{k=1}^K \frac{\lambda^{k-1}}{N} \{a_k\} \big(\widetilde \phi_k (a_k^-) \vee 0\big)
, \\
\cW^+ &:= \sum_{k=1}^K \frac{\lambda^{k-1}}{N}
\sum_{t=a_k^+}^N (\widetilde \phi_k(t) \vee 0)
- \sum_{k=1}^K \frac{\lambda^{k-1}}{N} \{a_k\} \big(\widetilde \phi_k (a_k^+) \vee 0\big)
\,.
\label{eq:tilted-p}
\end{align}
\end{definition}

\begin{lemma}
    \label{lem:tilting}
The law $\mu_{N,K}$ of \cref{eq:1d-sos-many-curves-phi} coincides 
with $\Q_\fh (\cdot \mid B)$ of \cref{def:Qh},   for the 
non-crossing event $B$ of \cref{eq:1d-sos-domain-many-phi}.  That is, 
\begin{align}\label{def:B}
B:=\Big\{ \underline{\widetilde \phi} (N)=\underline 0, \quad 
\widetilde \phi_1(t) \ge \widetilde \phi_2(t)
\ge \cdots \ge \widetilde \phi_{K}(t) \ge 0, \quad \forall t \in (-N,N) \Big\}\,.  
\end{align}
\end{lemma}
\begin{proof} The law $\mu_{N,K}$ of \cref{eq:1d-sos-many-curves} matches 
that of $\varphi_k(i)=\sum_{j=-N}^{i-1} \fX_k(j)$, for \abbr{iid} $\{\fX_k(j)\}$
of marginal law $\P_0$ of \cref{eq:rw-law},
restricted to non-crossing functions in a wedge (see \cref{eq:1d-sos-domain-many}),
and using the geometric area tilt weight
\begin{align}
\exp\Big(-\frac{\fh} N \sum_{k=1}^K \lambda^{k-1} \sum_{t=-N}^N \varphi_k(t) \Big) &= 
 \exp\Big(-\frac{\fh} N\sum_{k=1}^K\lambda^{k-1} \sum_{j=-N}^{N-1} (N-j)\fX_k(j)\Big)
\nonumber \\
&= \exp\Big(\sum_{k=1}^K \sum_{j=-N}^{N-1} \frac{\fh \lambda^{k-1}}{N} j \fX_k(j) \Big) \,,
\label{dfn:W-phi}
\end{align}
where the last identity in \cref{dfn:W-phi}
uses that by \cref{eq:1d-sos-domain-many-phi},
\begin{equation}\label{eq:ident-wedge}
\phi_k(N) = \varphi_k(N)=\sum_{j=-N}^{N-1} \fX_k(j) = 0 \,, \qquad \forall k \,.
\end{equation}
In view of \cref{eq:theta_k-def,def:akpm}, the Radon--Nikodym derivative for the change 
from marginal law $\P_0$ of the \abbr{iid} $\fX_k(j)$ to the marginal laws $\P_{\theta_k(j)}$ of the independent $\tilde \fX_k(j)$, transforms the weight \cref{dfn:W-phi}
to be of the form $e^{-\fh \widetilde{\cW}^-} e^{-\fh \widetilde{\cW}^+}$, where
\[
\widetilde{\cW}^- :=  \sum_{k=1}^K \frac{\lambda^{k-1}}{N}
\sum_{j=-N}^{a_k^- -1}  (-a_k - j) \tilde \fX_k(j)\,,\qquad
\widetilde{\cW}^+ := \sum_{k=1}^K \frac{\lambda^{k-1}}{N}
 \sum_{j=a_k^+}^{N-1} (a_k - 1 - j) \tilde \fX_k(j) \,.
\]
We now apply the summation by parts formula, whereby for any $x,y,a$ and $\psi$,
\[ \sum_{j=x}^{y-1}(\psi(j+1)-\psi(j))(a-j)=\psi(y)(a-y+1)+\psi(x)(x-a)+\sum_{j=x+1}^{y-1}\psi(j)\,;
\]
Combining this with the \abbr{rhs} of \cref{eq:ident-wedge},  which holds also for $\tilde \fX_k(j)$, 
we deduce that in terms of $\widetilde \varphi_k(i):= \widetilde \phi_k(i)+ (|i|-N)$, 
\begin{align*}
 \sum_{j=-N}^{a_k^- -1}  (-a_k - j) \tilde \fX_k(j)&=
\sum_{i=-N}^{a_k^-} \widetilde \varphi_k(i) + \widetilde \varphi_k(a_k^-) (a_k^+-a_k)
\,, \\
\sum_{j=a_k^+}^{N-1} (a_k- 1 - j) \tilde \fX_k(j)
& = \sum_{i=a_k^+}^N \widetilde \varphi_k(i) + \widetilde \varphi_k(a_k^+) (a_k^+-a_k)
\,.
\end{align*}
Substituting these identities in the expressions for $\widetilde{\cW}^\pm$ results in 
$\widetilde{\cW}^\pm=\cW^\pm$,  
thereby concluding the proof of the lemma.
\end{proof}
We proceed to show that for large $N$, the conditioning on $B$ of \cref{def:B}
can be relaxed to conditioning on $A \supset B$ whose structure 
assists us later in decoupling the relevant time intervals we
consider in \cref{thm:1}.
\begin{definition}\label{def:Irb}
Fix $\epsilon \in (0,\tfrac{1}{9})$, $\Delta=N^{2/3+\epsilon}$, $a_0^+=N+\Delta$ 
and set
\begin{align*}
% I_{\pm r}&:=\pm [a^+_{r} ,a^+_{r-1}-\Delta)\qquad\qquad (1 \le r \le K)\,,\\
\fI_{\pm r} &:=\pm [a^+_{r}- \Delta,a^+_{r-1}-\Delta)\qquad(1 \le r \le K)\,,\\
% I_{K+1} =
\fI_{K+1} &:= \;\;\, (a^-_K+\Delta,a^+_K-\Delta) \,.
\end{align*}    
Consider the event 
\begin{equation}\label{def:A}
A := \{ \underline{\widetilde \phi} (N)=\underline 0\} \cap A^- \cap A^+ \,,
\end{equation}
where
\begin{equation}
\label{def:Apm}
A^\pm := \big\{\widetilde \phi_{r} \ge \cdots \ge \widetilde \phi_{K} \ge 0
\quad \mbox{on} \quad \fI_{\pm r}, \quad \mbox{for} \quad 1 \le r \le K\big\} \,.
\end{equation}
Note that 
$\{\fI_r\}$ 
partitions $(-N,N)$ into $2K+1$ sub-intervals,
as depicted in \cref{fig:tilde-phi-events-A-B}.
\end{definition}

Indeed,  thanks to our next result (which is proved in \cref{sec:decup}),  
we can replace
$\mu_{N,K}(\cdot)
=\Q_{\fh}(\cdot\mid B)
$
by $\Q_{\fh} (\cdot \mid A)$,  
for the event $A$ of \cref{def:A}. 
\begin{proposition}\label{lem:non-crossing-BB}
For $\fh > \theta_\star$, $\lambda>1$ and $A$ of \cref{def:A},  
\begin{equation}\label{eq:non-crossing}
 \lim_{N \to \infty} \| \mu_{N,K}(\cdot) - \Q_\fh(\cdot\mid A) \|_\tv = 0 \,.
\end{equation}
\end{proposition}

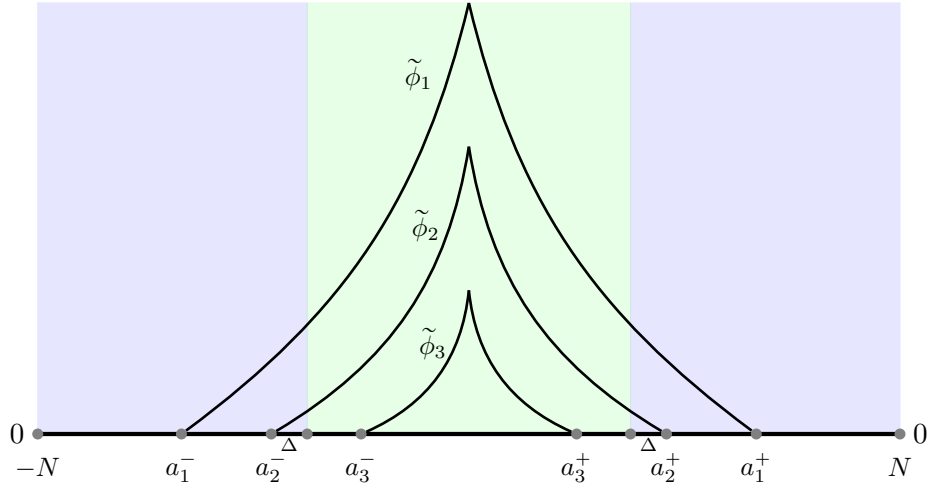
\begin{figure}
    \begin{tikzpicture}
    \begin{scope}[scale=0.95]
    \filldraw [blue,opacity=0.1] (-6,0)--(-2.25,0)--(-2.25,6)--(-6,6)--cycle;
    \filldraw [blue,opacity=0.1] (6,0)--(2.25,0)--(2.25,6)--(6,6)--cycle;
    \filldraw [green,opacity=0.1] (-2.25,0)--(2.25,0)--(2.25,6)--(-2.25,6)--cycle;
    %\draw [line width=3pt,green,opacity=0.5] (or)--(apD2);

    \draw[ultra thick, black] (-6,0)--(6,0);
    \node[circle,scale=0.4,fill=gray,label={[label distance=1pt]below:{\small$a_1^-$}}] (am1) at (-4,0) {};
    \node[circle,scale=0.4,fill=gray,label={[label distance=1pt]below:{\small$a_1^+$}}] (ap1) at (4,0){};
    \node[circle,scale=0.4,fill=gray,label={[label distance=1pt]below:{\small$a_2^-$}}] (am2) at (-2.75,0){};
    \node[circle,scale=0.4,fill=gray,label={[label distance=1pt]below:{\small$a_2^+$}}] (ap2) at (2.75,0){};
    \node[circle,scale=0.4,fill=gray,label={[label distance=1pt]below:{\small$a_3^-$}}] (am3) at (-1.5,0){};
    \node[circle,scale=0.4,fill=gray,label={[label distance=1pt]below:{\small$a_3^+$}}] (ap3) at (1.5,0){};

    \node[circle,scale=0.4,fill=gray] (amD2) at (-2.25,0){};
    \node[circle,scale=0.4,fill=gray] (apD2) at (2.25,0){};
    \node[label={\tiny$\Delta$}] () at (-2.5,-0.52){};
    \node[label={\tiny$\Delta$}] () at (2.5,-0.52){};
    \node[label={\small$\widetilde\phi_1$}] () at (-0.7,4.5){};
    \node[label={\small$\widetilde\phi_2$}] () at (-0.6,2.4){};
    \node[label={\small$\widetilde\phi_3$}] () at (-0.5,.75){};
    %\node[label={\tiny$\Delta$}] () at (2.5,-0.52){};

    \node[circle,scale=0.4,fill=gray,label={[label distance=-1pt]left:$0$},label={[label distance=3pt]below:{\small$-N$}}] (ol) at (-6,0) {};
    \node[circle,scale=0.4,fill=gray,label={[label distance=-1pt] right:$0$},label={[label distance=3pt]below:{\small$N$}}] (or) at (6,0) {};
    \draw[line width=1pt, black] (am1) to[bend right=20] (0,6);
    \draw[line width=1pt, black] (am2) to[bend right=25] (0,4);
    \draw[line width=1pt, black] (am3) to[bend right=30] (0,2);
    \draw[line width=1pt, black] (ap1) to[bend left=20] (0,6);
    \draw[line width=1pt, black] (ap2) to[bend left=25] (0,4);
    \draw[line width=1pt, black] (ap3) to[bend left=30] (0,2);
    \end{scope}
    \end{tikzpicture}
    \caption{Illustration of the setting of \cref{def:Irb}. The curves $\{\widetilde\phi_k\}$ are concentrated (after rescaling) about $\{\phi^\star_k\}$.  The event $B$ says, as per \cref{lem:tilting}, that $\widetilde\phi_k(i)\geq \widetilde\phi_{k+1}(i)\geq 0\; \forall i$, starting and ending at $0$. The event~$A$ relaxes this condition to $|i|> a_k^+ - \Delta$, which for $\widetilde\phi_2$ is the blue shaded region.}
    \label{fig:tilde-phi-events-A-B}
\end{figure}

In analyzing $\Q_\fh(\cdot \mid A)$,  it is useful to represent this conditional 
measure via the following product measure $\widecheck\Q_{\fh}(\cdot)$.
\begin{definition}\label{def:QI}
Let $\widecheck\Phi:=(\Phi^-,\Phi^o,\Phi^+)$,  for 
\begin{align}\label{dfn:Phim}
\Phi^- &:= \{\widecheck{\phi}^-_r(i), \qquad\qquad  \qquad \qquad \; i \in [-N,a_r^-+\Delta], \qquad \;\; 1 \le r \le K\} \,, \\
\Phi^o &:= \{ \widecheck{\phi}^o_r(i)
-\widecheck{\phi}^o_r(a_r^-+\Delta)
, \qquad i \in [a_r^-+\Delta,a_r^+-\Delta], \quad 1 \le r \le K \} \,,\label{dfn:Phis}
\\
\Phi^+ &:= \{ \widecheck{\phi}^+_r(i),\qquad \qquad \qquad \qquad
i \in [a_r^+ - \Delta,N], \qquad \quad  1 \le r \le K\} \,.
\label{dfn:Phip}
\end{align}
Let $\Q_{\fh}^{(-)}$ be the sub-probability measure on $\Phi^-$ 
whose Radon--Nikodym derivative 
\abbr{wrt} the $\Phi^-$-marginal 
of $\Q_\star$
is $e^{-\fh\cW^-} \one_{A^-}$ of \cref{eq:tilted,def:Apm},  denoting by 
$\widehat{\Q}_{\fh}^{(-)}$ the corresponding normalized probability measure.
We similarly denote by $\widehat{\Q}_{\fh}^{(+)}$
the law of $\{\Phi^-(-\,\cdot)\}$, and define the product measure
\[
\widecheck \Q_{\fh}(\widecheck\Phi):=\widehat{\Q}_{\fh}^{(-)}(\Phi^-)\widehat{\Q}_{\fh}^{(+)}(\Phi^+)\Q_\star(\Phi^o) \,.
\]
\end{definition}
\begin{remark}[Gibbs--Markov property]
\label{def-Gibbs}
We note that  $\widehat{\Q}_{\fh}^{(-)}$  can be written in the form
$$  \widehat{\Q}_{\fh}^{(-)}(\Phi^-)=\frac{1}{Z}\prod_{i=-N}^{a_K^-+\Delta-1} 
F_i\big(\underline{\widecheck{\phi}}^-(i),\underline{\widecheck{\phi}}^-(i+1)\big) 
$$
for appropriate functions $F_i$ and a normalizing constant $Z$.  As such, it satisfies the Gibbs--Markov property: for  $I=(i_L,i_R)$,  any $-N \le i_L \le i_R \le a_K^-+\Delta$,  we have that
$$  \widehat{\Q}_{\fh}^{(-)}(\Phi^-_I \mid \Phi^-_{I^c})=\frac{1}{Z'} \prod_{i=i_L}^{i_R-1} 
F_i\big(\underline{\widecheck{\phi}}^-(i),\underline{\widecheck{\phi}}^-(i+1)\big) \, 
\one_{E} \,,
$$
where $E$ is the event that $\Phi^-_{I^c}$ agrees with $\underline{\widecheck{\phi}}^-(i)$
at $i=i_L$ and $i=i_R$.  The same applies to  $\widehat{\Q}_{\fh}^{(+)}$ and (trivially) to $\Q_\star$.
\end{remark}
\begin{definition}\label{def:QIBis}
We denote by $\Q_{\fh}^I$ the probability measure induced by $\widecheck \Q_\fh$  on the collection of paths  $\Phi=\{\underline{\widetilde{\phi}}(i), |i| \le N\}$ through the transformation
\begin{equation}
\label{eq-270726a}
\widetilde \phi_r(i):= \left\{\begin{array}{ll}
 \widecheck \phi_r^-(i), & i \in [-N,a_r^-+\Delta],\\
\widecheck \phi^-_r(a_r^-+\Delta)+\widecheck \phi^o_r(i)-\widecheck\phi^o_r(a_r^-+\Delta),  & i \in [a_r^-+\Delta,a_r^+-\Delta),\\
\widecheck \phi^+_r(i), &i \in [a_r^+ - \Delta,N].
\end{array}
\right.
\end{equation}
\end{definition}
In words, the probability measure  $\Q_{\fh}^I$ is obtained by concatenating, for each $r$, three trajectories. The first collection of trajectories 
$\{ \widetilde \phi _r(i)\}_{i\in [-N, a_r^-+\Delta], r=1,\ldots,K}$ is obtained by starting with independent  random walks whose increments have the law of $\tilde{\fX}+1$ under $\P_{-\theta_\star}$ until $a_r^-$, continuing with independent increments of law as in  \cref{eq:tilde-phi-def} until $a_r^-+\Delta$, 
tilting the laws according to area tilt in $[-N,a_r^-]$, and imposing the ordering and positivity constraints \cref{eq:1d-sos-domain-many-phi}.  Similarly, the pieces 
$\{ \widetilde \phi _r(i)\}_{i\in [a_r^+-\Delta,N], r=1,\ldots,K}$ have the law of the time  and space reversal of the first piece. The middle piece consists of partial sums of independent increments as in 
 \cref{eq:tilde-phi-def},  where scaling time and space by $N$,  their mean effectively 
matches on $(-\alpha_r,\alpha_r)$ the strictly positive function $\phi_r^\star$
(see \cref{eq:phi-star-equiv}).  

The transformation \eqref{eq-270726a} does not use 
%the increments 
$\{ \widecheck{\phi}^o_r(a_r^+-\Delta) -\widecheck{\phi}^o_r(a_r^+-\Delta-1) \}$ 
from $\Phi^o$,  which are needed in order to determine the event $\Gamma$ of \cref{eq:w-y-pm}.
Thus,  while $\widecheck{\Q}_{\fh}(\Phi)=\Q_{\fh}^I(\Phi)$,  the quantity $\Q_{\fh}^I(\Gamma)$ 
is not well defined and in the next lemma we must use $\widecheck{\Q}_{\fh}$.
\begin{lemma}\label{lem:Q-A}
We have the identity
 \begin{equation}
\label{def:Gamma}
\Q_{\fh}(\Phi \mid A)= \widecheck \Q_{\fh}(\Phi \mid \Gamma)\,,\end{equation}
with the event
\begin{equation}
\label{eq:w-y-pm}
\Gamma=\{\underline{y}^{(+)}=\underline{y}^{(-)}+\underline{y}^{(o)}\}\,,
\end{equation}
where the $\Z_+^K$-valued $\underline{y}^{(\pm)} \subset \Phi^\pm$ and the $\Z^K$-valued $\underline{y}^{(o)}\subset \Phi^o$ are given by 
\begin{align}\label{eq:y-def} 
%\underline{y}^{(+)} - \underline{y}^{(-)} = 
y^{(o)}_r&:=
\widecheck{\phi}_r^o(a_r^+-\Delta)-\widecheck{\phi}^o_r(a_r^-+\Delta),\nonumber\\
y^{(+)}_r&:= \widetilde{\phi}_r(a_r^+-\Delta)\,, \quad 
y^{(-)}_r:=\widetilde{\phi}_r(a_r^-+\Delta)\,, \quad 1 \le r \le K\,. 
\end{align}
\iffalse
That is, for any $\Phi$ satisfying \cref{eq:w-y-pm} we have that 
\begin{equation}\label{def:Gamma-old}
\Q_{\fh}(\Phi \mid A) = \frac{1}{\Gamma} \Q^I_{\fh}(\Phi) \quad 
\hbox{where} \quad 
\Gamma := \Q_{\fh}^I(\underline{w}=\underline{y}^{(+)}-\underline{y}^{(-)})\,.
\end{equation}
\fi
\end{lemma}
\begin{proof} Under $\widecheck{\Q}_{\fh}$, the part $\Phi^+$ starts at 
$\widecheck{\underline\phi}^+(N)=\underline{0}$,
\emph{terminating as specified by $\underline{y}^{(+)}$}.
Thanks to the time-reversal symmetry of the weight 
$\cW^- + \cW^+$, the event $A$ and our tilting $\{\theta_r(j)\}$,
this time-reversed construction of $\Phi^+$ according to the 
weight $e^{-\fh\cW^-} \one_{A^-}$, matches the sub-probability measure 
induced by $\Q_\star$ with weight $e^{-\fh \cW^+} \one_{A^+}$ when 
starting at the positions $\underline{y}^{(+)}$ and 
ending at $\underline{\widetilde \phi}(N)=\underline{0}$. Moreover, with 
$\Q_{\fh}^{(-)}$ considered at a single 
non-random initial condition, its normalizing factor (en-route 
to $\widehat{\Q}_{\fh}^{(-)}$), cancels out in \cref{def:Gamma}. Further, both $A$ and $\cW^-+\cW^+$ 
are determined by $\Phi^\pm$, thus allowing us
to independently construct $\Phi^o$ as in \cref{def:QI}
(thanks to the invariance of $\Q_\star(\cdot)$ to global 
spatial translations). Finally, recall from
\cref{def:A} that $A = A^+ \cap A^- \cap \{\underline{\widetilde{\phi}}(N)=\underline{0}\}$. 
%which is equivalent to \eqref{eq:w-y-pm}.
Having incorporated the constraint $A^- \cap A^+$ already in~$\Q_\fh^I$,
when given $A$ we need only require the compatibility condition $\Gamma$
of \cref{eq:w-y-pm}, thereby establishing \cref{lem:Q-A}.
\end{proof}

\begin{definition}
\label{def-RWGATE}
Given $1 \le r \le K$,  $\Z^{K+1-r}_+$-valued $\underline u,\underline v$ and an interval $I$,  the law of 
a \abbr{rw} \abbr{gate} consisting of $K+1-r$ curves $\{\widetilde \phi_k(t), t\in I\}$,
 $r\leq k\leq K$,
% with left and right boundaries $\underline u,\underline v$, 
 is obtained by random walks starting at $\underline u$ 
 with increments of law $\tilde{\fX}+1$ under $\P_{-\theta_\star}$, 
 area-tilted with parameters $\frac{\fh}{N} \lambda^{k-1}$,  
 conditioned on being ordered,  non-negative,  and ending at $\underline v$.
\end{definition}

\begin{remark}\label{rem:Qhat} 
From the Gibbs--Markov property \cref{def-Gibbs},  for $1 \le r \le K$,  restricting $\widehat{\Q}^{(-)}_{\fh}$ to 
$I_{-r}:=(a^-_{r-1} +\Delta,a^-_{r}]$ conditional on 
\begin{align*}
\underline{u}^{(r)} := \big(\widetilde{\phi}_{r}(a^{-}_{r-1}+\Delta),\ldots,\widetilde{\phi}_{K}(a^-_{r-1}+\Delta)\big)
\;\; \mbox{and} \;\; 
\underline{v}^{(r)} := \big(\widetilde{\phi}_{r}(a^{-}_{r}),\ldots,\widetilde{\phi}_{K}(a^-_{r})\big)
\end{align*}
gives \abbr{rw} \abbr{gate} on $I_{-r}$   
with the specified $\underline{u}^{(r)}$ left boundary 
and $\underline{v}^{(r)}$ right boundary.
\end{remark}

\section{Tools: stochastic ordering and control of an area-tilted rw}\label{sec:prelim}

As in \cite[Sec.~3.2]{HKS25} we denote by $\P_{1,I;g,h}^{b,1;u,v}$ the law of a single $\frac{b}{N}$-area-tilted 
random bridge on time interval $I$,  starting at height $u$ on the left,  ending at height $v$ on the right,  
constrained to be in $[h(x),g(x)]$ at time $x \in I$.  We similarly use $\P_{1,I;g,h}^{b,1;u,\star}$ for free (unspecified) right boundary,  setting $g=\infty$ or $h=-\infty$ if no ceiling/floor.
%  and $b=0$ for random walk/bridge without any tilt.
The stochastic ordering of such laws for $K \ge 1$ ordered area-tilted bridges and certain 
parameters ($g, h, b, u, v$) is provided in \cite[Lem.~3.2]{HKS25}.  We state a more general result,  allowing free boundaries,  general area-tilt parameters and independent but possibly non-identical increments,  omitting its
proof which follows the same coupling argument of \cite{HKS25}.

\begin{lemma}[extension of {\cite[Lem.~3.2]{HKS25}}]
    \label{lem:Qh-monotonicity}
    Consider laws $\Q^\uparrow,\Q^\downarrow$
    on 
    % ensembles of  $K \ge 1$ 
    ordered random walk curves $X_1 \ge \cdots \ge X_K$ on interval $I\subset \Z$, 
    parametrized by curves $g(x) \ge h(x)$ on $I$,  
    boundary values $\underline{u}, \underline{v} \in \Z^K$ and 
    area-tilt $\underline{b}=\{b_i(x),  x \in I,  1 \le i \le K\}$,  such that:
    \begin{itemize}
    \item The ordered curves $X_1(x)\geq \cdots \geq X_K(x)$ for all $x\in I$ have a global ceiling 
    $g(\cdot) \le \infty$ and floor $h(\cdot) \ge - \infty$ (i.e.,  $X_1 \leq g$ and $X_K \geq h$).
    \item The curves have boundary conditions $\underline{u},\underline{v}$, each of which is either ordered (e.g., $\underline{u}=(u_1,\ldots,u_K)$ where $u_i\geq u_{i+1}$) or free.
    \item The law of the increment $\mathfrak{X}\in \Z$ for curve $1\leq i\leq K$ at location $x\in I$ may depend on $i,x$ as long as it is of the form $\P(\fX_{i,x}=k) = \exp(-f_{i,x}(k))$ for a convex function $f_{i,x}$ which is finite at
 least when $k \in \pm 1$.
    \item The area tilt of curve $1\leq i\leq K$ at location $x\in I$ may depend on $i,x$,  of the form $\exp(-b_{i}(x) X_i(x))$.
    \end{itemize} 
    If the parameters of $\Q^\uparrow,\Q^\downarrow$ are such that
    \begin{enumerate}[(a)]
    \item the ceilings and floors are ordered: $g^\uparrow\geq g^\downarrow$ and $h^\uparrow\geq h^\downarrow$ coordinate-wise; \item the boundary conditions are ordered: $\underline{u}^\uparrow\geq \underline{u}^\downarrow$ coordinate-wise or both are free, the same applies to $\underline{v}$;
    \item the area tilts are reversely-ordered:  $b_i^\uparrow(x) \leq b_i^\downarrow(x)$ for all $i,x$;
    \end{enumerate} 
    then there is a coupling of $\Q^\uparrow,\Q^\downarrow$  so that $X_i^\uparrow(x)\geq X_i^\downarrow(x)$ for all $1\leq i\leq K$ and $x\in I$.
\end{lemma}
\begin{proof}
The proof is identical to the one given as in \cite[App.~A]{HKS25}, and in what follows we briefly explain why it extends to the setting given here. In said proof, $\Q^\uparrow,\Q^\downarrow$ were realized as the stationary distributions of two instances of Glauber dynamics coupled via the standard monotone coupling, and the sought result followed from verifying that the monotone coupling is valid, i.e., it maintains the pointwise order $X_i^\uparrow(x)\ge X_i^\downarrow(x)$ for all $1\leq i\leq K$ and $x\in I$. This verification relied on only three properties,
each guaranteed by our hypotheses when evaluating whether to modify $X_i(x)$ by $\{-1,0,1\}$:
\begin{itemize}
\item The convexity of each $f_{i,x}$ makes
the update weights log-concave; since the increment laws may depend on $(i,x)$ but
are always tested for the same pair between the two systems, having non-identical increments changes
nothing. 
\item A log-concave law truncated to the window $[ h(x) \;\vee\; X_{i+1}(x), g(x)\;\wedge\; X_{i-1}(x)]$ is stochastically
nondecreasing in each truncation endpoint.  Therefore,  raising the floor, the ceiling, or the
neighboring curve,  as in hypothesis~(a),  together with the non-crossing constraint coupling
the curves,  raises the conditional law stochastically. 
\item The area tilt enters only through the log-linear factor $e^{-b_i(x) X_i(x)}$ and such a re-weighting
shifts the law stochastically downward as its coefficient grows,  (hence the reverse-ordering in~(c)).  
This argument accommodates arbitrary site-dependent
coefficients $b_i(x)$, in place of the geometric choice $\tfrac{a}{N}b^{i-1}$ of
\cite{HKS25}. 
\end{itemize} 
Consequently, in the generalized setup,  the coupling maintains the pointwise order $X_i^\uparrow(x)\ge X_i^\downarrow(x)$. A free boundary requires no comparison of endpoint
data: the endpoint is nothing but another updated site, so hypothesis~(b)---each pair $(\underline{u}^\uparrow,\underline{u}^\downarrow)$ and $(\underline{v}^\uparrow,\underline{v}^\downarrow)$ can be either
ordered or free---suffices, as in the case of a specified boundary.
\end{proof}

Note that the increments of \cref{eq:tilde-phi-def} are of the form specified above,  as is the partial area-tilt 
in $\widehat{\Q}_\fh^{(-)}$.  We thus utilize \cref{lem:Qh-monotonicity} to compare various laws induced by it,
supplemented by the following upper tail on the height under $\P_{1,I;\infty,0}^{b,1;0,\star}$
at the free right-end of a large interval
(here the increments will be from $\P_{-\theta_\star}(\cdot-1)$ of \cref{eq:P-theta}).
\begin{lemma}\label{lem:Ofer-A4}
Fix $b,\delta,\eta>0$ and consider the area-tilted sequence $\{Y_i,  i \in I=[0,k]\}$ 
with $\Z$-valued \abbr{iid} increments 
% $\{Y_i-Y_{i-1}\}$
from a zero-mean,  log-concave law of 
finite exponential moments and positive probabilities of $\{-1,0,1\}$-valued increments. 
For $C<\infty$,  $\kappa>0$ depending on $b,\delta,\eta$ and any $k \in [\eta N,2N]$,
\[
\P_{1,I;\infty,0}^{b,1;0,\star} \big( Y_k \ge N^{1/3+\delta} \big) \le C e^{- N^{\kappa}} \,. 
\]
\end{lemma}
\begin{proof} For convenience we 
shift space up by one,
denoting by $\E^u$ the expected value starting at $Y_0=u \in \N$.  
Unraveling the definition of $\P_{1,I;\infty,0}^{b,1;0,\star}$ 
we merely need to show that for  $c'=c'(b,\eta)>0$ and all large $N$,  
\begin{equation}\label{eq:ratio}
\frac{\E^1\big[\one_{\{Y_k\ge x N^{1/3}\}}   \one_{\cA_k}  e^{-\cW_k} \big]}
{\E^1\big[\one_{\cA_k} e^{-\cW_k} \big]} \le e^{-c' x} + e^{-c' N^{1/3}} \,,
\end{equation}
uniformly over $k \in [\eta N,2N]$ and $(\log N)^2 \le x \le N^{1/3}$,  where
\[ 
\cA_k:=\{Y_i \ge 1 \quad \forall\, 1 \le i \le k\},   \qquad \cW_k := \frac bN \sum_{i=1}^k Y_i \,.
\]
Indeed,  from \cite[(3.3)]{ISV15},  see also \cite[Prop.~4.1]{HKS25},  for some 
$c^\star=c^\star(b,\eta)$ finite 
\begin{equation}\label{eq:confine}
\E^1\big[\one_{\cA_k} e^{-\cW_k} \big] \ge (c^\star)^{-1} e^{-c^\star N^{1/3}} \,.
% = k (N/b)^{-2/3} 
\end{equation}
Now,  by \cref{lem:Qh-monotonicity} we increase the probability of 
$\cC^c_k :=\Big\{\max_{i \le k} \{Y_i\} > M \Big\}$ upon taking $b=0$. 
Consequently,  for $M :=\lceil \ell N^{2/3} \rceil$ and some $\ell$ finite,  
we get by standard random-walk tail bounds that for $k,N$ as above,  
\[
\frac{\E^1 [\one_{\cA_k \cap \cC^c_k} e^{-\cW_k} ]}{\E^1[\one_{\cA_k} e^{-\cW_k}]}   
\le \frac{\P^1 (\cA_k \cap \cC^c_k)}{\P^1(\cA_k)}
\le e^{-2 c^\star N^{1/3}}  \,.
\]
Thus,  \cref{eq:confine} holds true with $\cC_k \cap \cA_k$ instead of $\cA_k$ and setting the sub-probability 
transition 
% kernel 
\begin{equation}\label{dfn:hat-T}
T(u,v):= \P_{-\theta_\star}(v-u-1)  e^{-\frac bN v} \quad \hbox{on} \quad u,v \in \Omega_N := \{1,\ldots,M\} \,,
\end{equation}
it suffices for \cref{eq:ratio} to show that
\begin{equation}\label{eq:ratio-transfer}
\frac{\sum_{y\ge xN^{1/3}} T^{k}(1,y)}{\sum_{y} T^{k}(1,y)} =
\frac{\E^1\big[\one_{\{Y_k\ge x N^{1/3}\}}   \one_{\cA_k \cap \cC_k}  e^{-\cW_k} \big]}
{\E^1\big[\one_{\cA_k \cap \cC_k} e^{-\cW_k} \big]} 
\le e^{-c' x}  \,.
\end{equation}
More generally,  we associate with any probability measure $\nu$ on $\Omega_N$ the laws
\begin{equation}\label{dfn:wk-nu}
w_k^{\nu}(y) := \frac{(\nu T^k)(y)}{\sum_{y'} (\nu T^k)(y')}\,,
\end{equation}
of $Y_k$ under the length-$k$ confined area-tilted measure started from $Y_0 \sim \nu$.
In particular,  the \abbr{lhs} of  \cref{eq:ratio-transfer} is precisely $w_k^{\delta_1}([xN^{1/3},M])$.

While the operator $T$ of non-negative entries is not self-adjoint,  
it is irreducible and aperiodic,  hence having a Perron--Frobenius (maximal) eigenvalue $\rho_N$,
with an associated left 
% Perron-Frobenius 
eigenvector $\psi$ of strictly positive entries.  Normalizing $\psi$ to be a probability measure on $\Omega_N$
results in $w_k^\psi=\psi$ and we establish \eqref{eq:ratio-transfer} upon showing that:\\
{\bf I.}  For some $c_1>0$,  all $N$,  $k$ and $x$ as above, 
\begin{equation}
\label{eq-230726a}
 w_k^{\psi}([xN^{1/3},M]) = \sum_{y=xN^{1/3}}^M \psi(y) \leq e^{-c_1 x} \,.
\end{equation}
{\bf II.}  Adapting a coupling argument from \cite[Section~3.4]{ISV15} yields for some $c''=c''(\eta)>0$, 
\begin{equation}
\label{eq-230726b}
\|w_k^{\delta_1}-w_k^{\psi}\|_\tv\leq e^{-c'' N^{1/3}} \,,   \qquad \forall k \in [\eta N,2N] \,.
\end{equation}
{\bf I.}  Towards \cref{eq-230726a},
we first show that for all $N$,
\begin{equation}
\label{eq-230726c}
\rho_N^{N^{2/3}}\geq c_2 > 0 \,.
\end{equation}
To see \cref{eq-230726c},  denote by $P^x$ the sub-probability law on $\Omega_N$ at  $b=0$,
starting at $x \in \Omega_N$.  Note that for 
$\mathcal Q_j =\bigcap_{i \le j} \{ N^{1/3} \le Y_i \le 5N^{1/3} \}$ and any $j \ge 1$, \begin{align}
\label{eq-230726d}
\rho_N^j= \rho_N^j \sum_y \psi(y)=\sum_{x,y} \psi(x) T^j(x,y) &\geq \psi(3N^{1/3}) \sum_y T^j(3N^{1/3},y)\nonumber\\
&\geq \psi(3N^{1/3}) P^{3N^{1/3}} (\mathcal Q_j) e^{-cjN^{-2/3}}
\end{align}
(with the exponential term bounding the maximal area tilt under $\mathcal Q_j$).
Further,  setting $I:=[2N^{1/3}, 4N^{1/3}]$,  by the invariance principle we have 
for some $p>0$ and all $N$, 
\[
\min_{x\in I}P^x(\max_{i\leq N^{2/3}}|Y_i-3N^{1/3}|<2 N^{1/3}, Y_{N^{2/3}}\in I)\geq p\,.
\]
Hence,  by the Markov property $P^{3N^{1/3}}(\mathcal Q_{\ell N^{2/3}})\geq p^{\ell}$,  so by
\cref{eq-230726d} at $j=\ell N^{2/3}$, 
\[
(\rho_N)^{\ell  N^{2/3}} \geq (pe^{-c})^\ell \psi(3N^{1/3}) \,.
\]
Taking the $\ell$-th root and $\ell \to\infty$,  yields \cref{eq-230726c} with $c_2=pe^{-c}$.

Next,  expressing 
$\{Y_{N^{2/3}-i}\}$ via  
the $\Z$-valued \abbr{rw} $\{\widetilde Y_i\}$ having the 
increment law of $Y_{i-1}-Y_i$  (namely, 
$-(\tilde{\fX}+1)$ under $\P_{-\theta_\star}$), 
yields the identity
\[
\sum_{u=1}^M T^{N^{2/3}} (u,v) = \E^v\Big[\prod_{i=1}^{N^{2/3}} \one_{\Omega_N}(\widetilde Y_i)
\exp(-\frac bN\sum_{i=0}^{N^{2/3}-1} \widetilde Y_{i}) \Big] \,.
\]
We thus have that for some $C_1,C_2$ finite,  all $N$ and any $v \in \Omega_N$,
\begin{align}
 e^{b v/N^{1/3}} & \sum_{u=1}^M  T^{N^{2/3}} (u,v)  \leq 
 \E^v \Big[\exp\big( -\frac{b}N \sum_{i=1}^{N^{2/3}-1} (\widetilde Y_{i}- v) \big) \Big] 
 = \E^0 \big[\exp ( -\frac{b}N \sum_{i=1}^{N^{2/3}-1} \widetilde Y_{i}  ) \big] \nonumber \\
 &=
 %\E^0 \Big[\exp(\cW_{N^{2/3}-1})\Big] =
 \exp\big( \sum_{j=1}^{N^{2/3}-1} 
\big\{ \Lambda(\frac{b j}{N}-\theta_\star) - \Lambda(-\theta_\star) + \frac{bj}{N} \big\} \big) \le 
\exp \big(C_1 \sum_{j=1}^{N^{2/3}} \frac{j^2}{N^2} \big) \le C_2 
\label{eq:mass-decay}
\end{align}
(since $\Lambda'(-\theta_\star)=-1$ and $\Lambda''(-\theta_\star)$ is finite).  With $\|\psi\|_\infty \le 1$, 
by \cref{eq-230726c} and \cref{eq:mass-decay},
\begin{equation}
\label{eq-230726e}
\psi(v)=\rho_N^{-N^{2/3}} \sum_{u=1}^M \psi(u) \,  T^{N^{2/3}}(u,v) \le \frac{1}{c_2} \sum_{u=1}^M 
T^{N^{2/3}}(u,v) \leq \frac{C_2}{c_2} e^{-b v/N^{1/3}}\,.
\end{equation}
Summing this bound over $v\ge xN^{1/3}$ establishes \cref{eq-230726a},  where for $x\ge(\log N)^2$
the polynomial prefactor has been readily absorbed by taking $c_1<b$.

\noindent{\bf II.}
We follow the coupling
argument of \cite[Section~3.4]{ISV15}.\footnote{Specifically,  we adapt \cite[Lem.~3 and Prop.~6]{ISV15}, whose
proofs use only the irreducibility, zero mean and finite exponential moments
of the increments.}
%their Prop.~5, which concerns
%bridge measures and the stationary ground-state chain---our \cref{eq:tv-bound}
%substitutes for it.}. 
First,  recall from \cref{dfn:hat-T} that $T(u,u)>0$,  hence also $T^\ell(u,u)>0$ for all $\ell \ge 1$.  
Let $h_k(u)=1$ and 
%we will represent $\bP^{\nu}$ as a Markov chain. Indeed, 
set for $j=0,\ldots,k-1$,  the positive vectors and Markov transition probabilities
\begin{equation}\label{dfn:pj}
h_j(u) :=\sum_{y\in\Omega_N} T^{k-j}(u,y)\,,  \qquad \pi_j(u,v):=\frac{1}{h_j(u)} T(u,v) \,h_{j+1}(v) \,.
\end{equation}
We further associate with each probability measure $\nu$ on $\Omega_N$ the initial law
\begin{equation}\label{dfn:nu-hat}
\hat \nu(y)= \frac{\nu(y) h_0(y)}{\sum_{u} \nu(u) h_0(u)} = \frac{\nu(y) h_0(y)}{\sum_{y'} (\nu T^k)(y')} \,,
\end{equation}
denoting by $\bP^\nu$ the law of the inhomogeneous Markov chain $\{Y_0,\ldots,Y_k\}$
of transition kernels $\{\pi_j\}$ and initial law $\hat \nu$.  In particular,  $\hat \nu=\nu$
when $\nu=\delta_1$,   whereas if $\nu=\psi$ then $\hat \nu=\rho_N^{-k} \psi (\cdot) h_0(\cdot)$.
It is also easy to verify that $\bP^\nu$ matches the law of $\{Y_i,  i \le k\}$ under the
area-tilting weight $\one_{\cA_k \cap \cC_k} e^{-\cW_k}$,  starting at $Y_0 \sim \nu$,  hence 
in view of \eqref{dfn:wk-nu},  
\[
w_k^\nu(y) = \bP^\nu(Y_k=y) \,,  \qquad \forall y \in \Omega_N\,.
\]
In particular,  with $\bQ$ the coupling where chains $(Y_i)\sim\bP^{\delta_1}$ and $(Y'_i)\sim\bP^{\psi}$
run independently until the stopping time $\tau:=\min\{i:Y_i=Y'_i\}$ and thereafter follow the same path,  
we have by definition of $\|\cdot\|_{\tv}$,  that
\begin{equation}\label{eq:coupling-ineq}
\big\|w_k^{\delta_1}-w_k^{\psi}\big\|_\tv \leq
\bQ(\tau > k)\,.
\end{equation}

Proceeding to establish  \cref{eq-230726b} by bounding the \abbr{rhs} of 
\eqref{eq:coupling-ineq},  
for any $0 \le s < t \le k$,  
\begin{equation}\label{eq:bP-bridge} \bP^\nu\big(Y_i = y_i\ \forall s < i < t\;\big|\; Y_i = y_i\ \forall i \in [0,s]\cup [t,k] \big) \propto \prod_{i=s+1}^{t} T(y_{i-1},y_i)\,,
\end{equation}
which in particular is independent of $\nu$ and the $h_j$'s.  This confined area-tilted \abbr{rw} bridge is the
measure to which the estimates of \cite{ISV15} are applied in the sequel.
To this end,  partition $[0,k]$ into the intervals
\[ I_j := \llb (j-1)N^{2/3}, jN^{2/3}\rrb \quad\mbox{for}\quad j = 1,\ldots, m:=\lfloor kN^{-2/3}\rfloor\ge\frac\eta2
N^{1/3}\]
plus a remainder final interval, and group them into $\lfloor m/3\rfloor$ triples
$(I_{3j},I_{3j+1},I_{3j+2})$. Fix $\varrho=\varrho(b,\eta)<\infty$,  to
be chosen later, and consider the following definition of the exceptional intervals $I_j$ where the \emph{minimum} of $Y_i$ (resp., $Y'_i$) exceeds $\varrho N^{1/3}$:
\begin{equation}\label{eq:high-intervals} \cB^{\delta_1} := \Big\{j\,:\; \min_{i\in I_j} Y_i > \varrho N^{1/3}\Big \}\,,\qquad
\cB^{\psi} := \Big\{j\,:\; \min_{i\in I_j} Y'_i > \varrho N^{1/3}\Big \}\,.
\end{equation}
We claim that, for both $\nu=\delta_1$ and $\nu=\psi$, if $\varrho$ is taken large enough then
\begin{equation}\label{eq:few-high}
\bP^{\nu} \left(|\cB^\nu| \geq m/48\right)  \leq e^{-N^{1/3}}\,.
\end{equation}
To see this,  recall that for any event $E$, 
\begin{equation}\label{eq:dfn-Pnu-E}
\bP^\nu(E)=\frac{\sum_{u}\nu(u) \, \E^u\!\left[\one_{E\cap \cA_k\cap\cC_k} e^{-\cW_k}\right]}
{\sum_{u}\nu(u) \, \E^u\!\left[\one_{\cA_k\cap\cC_k} e^{-\cW_k}\right]}\,,
\end{equation}
yielding our claim \eqref{eq:few-high},  once we show that the denominator in \eqref{eq:dfn-Pnu-E}
is at least $e^{-C' N^{1/3}}$ for some $C'(b,\eta)<\infty$,  
since on $E=\{ |\cB^\nu| \geq m/48\}$ one has that 
\[
\cW_k\ge \frac bN \cdot \frac
m{48}\,N^{2/3}\,\varrho N^{1/3}\ge \frac{b\varrho\eta}{96} N^{1/3}\,.
\]
For the required lower bound on $\E^1\!\left[\one_{\cA_k\cap\cC_k} e^{-\cW_k}\right]$ 
see \cref{eq:confine} (intersected with $\cC_k$),  
while for $\nu=\psi$ the denominator in \eqref{eq:dfn-Pnu-E} is 
$\sum_{v} (\psi T^k) (v) = \rho_N^k \ge c_2^{2N^{1/3}}$ by \cref{eq-230726c}.

Consider a triple $(I_{3j},I_{3j+1},I_{3j+2})$ whose left and right intervals do not belong to the aforementioned exceptional sets of intervals:
\begin{equation}\label{eq:promising-triple} 3j\notin \cB^{\delta_1}\cup\cB^{\psi}\,,\qquad 3j+2\notin \cB^{\delta_1}\cup\cB^{\psi}\,,\end{equation}
and call its middle interval $I_{3j+1}$ \texttt{good} (\`a la the $\eta$-\texttt{good} intervals of
\cite[Sec.~3.4]{ISV15}) if
\begin{itemize}
\item both $(Y_i)$ and $(Y'_i)$ lie in $[1,2\varrho N^{1/3}]$ throughout $I_{3j+1}$;
\item both $(Y_i)$ and $(Y'_i)$ are at most $\varrho N^{1/3}$ at the two endpoints of $I_{3j+1}$. 
\end{itemize}
We argue that there exists $\kappa_0(b,\varrho)>0$ such that, for any triple satisfying \cref{eq:promising-triple},
\begin{equation}\label{eq:good-interval}
\big(\bP^{\delta_1}\otimes\bP^{\psi}\big)\left(I_{3j+1} \mbox{ is \texttt{good}}\right) \geq \kappa_0^2\,.
\end{equation}
It suffices to show that the path $(Y_i)$ satisfies the above two properties with probability at least $\kappa_0$; the same exact reasoning will apply to the independent path $(Y'_i)$. 
Let $\ell_j$ be the \emph{first} time in $I_{3j}$ that $Y_i$ is at most 
$\varrho N^{1/3}$, and let $r_j$ be the \emph{last} time in $I_{3j+2}$ that $Y_i$ is at most $\varrho N^{1/3}$
(which are well-defined in view of \cref{eq:promising-triple}).  Conditioning on $(Y_i)$ outside the open window $(\ell_j,r_j)$, and recalling \cref{eq:bP-bridge}, the restriction of $(Y_i)$ to $\ell_j < i < r_j$ is an area-tilted \abbr{rw} bridge on $\Omega_N$ of length $r_j-\ell_j \in [N^{2/3},3N^{2/3}]$ and some 
endpoints $u,v\in [1,\varrho N^{1/3}]$.  We claim that the probability of this bridge staying within $[1,2\varrho
N^{1/3}]$ while being at most $\varrho N^{1/3}$ at both
endpoints of $I_{3j+1}$ is at least $\kappa_0>0$,  uniformly in the window data (namely, 
$\ell_j,r_j,u,v$).  Indeed,  on the confinement event the area tilt collected along the window is at
most $\frac b N\cdot 3N^{2/3}\cdot2\varrho N^{1/3}=6b\varrho$,  so the probability in
question is at least $e^{-6b\varrho}$ times the probability that an untilted ($b=0$) 
bridge conditioned to stay in $[1,M]$ is
confined to $[1,2\varrho N^{1/3}]\subseteq[1,M]$ and is at most
$\varrho N^{1/3}$ at the end points of $I_{3j+1}$.  The latter probability is
bounded from below,  uniformly over the relevant values of $\ell_j,r_j,u,v$,
% the window lengths, endpoint heights and marked times in question, 
by the invariance principle for \abbr{rw} bridges conditioned to stay positive
\cite[Theorem~2.4]{CaravennaChaumont2013}. This establishes \cref{eq:good-interval}. 

The event $|\cB^{\delta_1}| + |\cB^{\psi}| \leq \frac{m}{24}$,  which by \cref{eq:few-high} 
has probability at least $1- 2e^{-N^{1/3}}$,  implies that at least 
$
%\lfloor m/3\rfloor -m/24\geq 
\frac{m}{4}$ triples satisfy \cref{eq:promising-triple},  whence by \eqref{eq:good-interval}
the number of \texttt{good} intervals 
$I_{3j+1}$ stochastically dominates a $\Bin(\lceil m/4\rceil,\kappa_0^2)$ 
% random
variable.  We thus conclude 
%from Chernoff's bound
 that
\begin{equation}
\label{eq:good-count}
\big(\bP^{\delta_1}\otimes\bP^{\psi}\big)\left(\#\{j: I_{3j+1} \mbox{ is \texttt{good}}\} < \tfrac{\kappa_0^2}{8}\,m\right) \leq e^{-\kappa_0^4 m/50}+2e^{-N^{1/3}}\,.
\end{equation}
Next, conditionally on the values of both paths outside the \texttt{good} intervals,  
the interiors of their distinct middle blocks are independent and within each \texttt{good} 
interval the paths form two independent area-tilted length-$N^{2/3}$ bridges, 
with all four endpoint heights in $[1,\varrho N^{1/3}]$,  
each conditioned to stay in $[1,2\varrho
N^{1/3}]$.
%  (the ceiling at $M\ge2\varrho N^{1/3}$ is again automatically satisified). } 
The area tilt collected in such a confined block is at most $2b \varrho$, so 
removing the tilt at the cost of a factor $e^{-2b\varrho}$,  
reduces the meeting probability to that of \cite[Prop.~6]{ISV15}: it is shown there that for any irreducible zero-mean step law with finite exponential moments,  
there exists $p=p(\varrho)>0$ such that
within each \texttt{good} block,  there exists some $i\in I_{3j+1}$ such that $Y_i=Y'_i$ with
conditional probability at least $p$.  Setting  $p_0:=p e^{-2b\varrho}$ and recalling that 
% conditionally independently across the \texttt{good} blocks. 
$m \ge \frac\eta2 N^{1/3}$,  upon combining this with \cref{eq:good-count} we find that for all large $N$, 
\[
\bQ(\tau > k) \leq
(1-p_0)^{\kappa_0^2 m/8}+e^{-\kappa_0^4 m/50}+2e^{-N^{1/3}} \leq e^{-c'' N^{1/3}}\,.
\]
In view of \eqref{eq:coupling-ineq} this proves \cref{eq-230726b} and 
thereby \cref{eq:ratio-transfer} and the lemma.
\end{proof}

\section{Curves near separation and decoupling}\label{sec:separation}
Combining \cref{lem:Qh-monotonicity,lem:Ofer-A4} we control 
various heights under $\widehat{\Q}_\fh^{(-)}$,  as follows.
\begin{proposition}\label{prop:N-third-almost-tight}
For some $C_r=C_r(L,\delta)$ finite, any $N,  L<\infty$ and $\delta>0$: 
\begin{equation}\label{eq:almost-tight-induction}  \widehat\Q_\fh^{(-)}\Big(\sup_{-N\leq x \leq a_r^-} \widetilde\phi_r(x)> (K+1-r) N^{1/3+2\delta}\Big) \le C_r N^{-L}\,, 
\qquad \forall 1 \le r \le K \,.
\end{equation}
\end{proposition}
\begin{proof}  We start at the bottom-curve $\widetilde\phi_K$,  defined under
$\widehat{\Q}_\fh^{(-)}$,  as in \cref{def:QI}, 
%and \cref{rmk:independce},  
on the union of disjoint intervals $I'_1:=[-N,a_{K-1}^-+\Delta]$,
$I'_2:=(a_{K-1}^-+\Delta,a_K^-]$ and $I_3'=(a_K^-,a_K^-+\Delta]$.
The curve $\widetilde\phi_K$ is conditioned to be non-negative, having 
over $I_1' \cup I_2'$ \abbr{iid} increments of law $\tilde{\fX}+1$ under $\P_{-\theta_\star}$, of zero mean and finite exponential moments,
and an area tilt, with additional ceiling on $I_1'$ by the curve $\widetilde\phi_{K-1}$. The curve $\widetilde\phi_K$ has neither tilt, 
nor ceiling over $I_3'$ where it is the partial sum, as in
\cref{eq:tilde-phi-def}, of independent, positive mean 
increments based on $\tilde{\fX}_{K}(j) \sim \P_{\theta_K(j)}$. 
Setting for $v \in \Z_+$,
\[
q_K(v):=\P(\inf_{j \in I_3'} \{\widetilde\phi_K(j)\} > -1  \mid  \widetilde \phi_K(a_K^-)=v) \,,
\]
note that $q_K(v) \ge q_K(0)$ which is further reduced upon 
subtracting the positive mean from each of the increments
of $\widetilde \phi_K$ in $I_3'$.
%$\tilde{\fX}_{K}(j)-\sgn(j)$ 
Doing so yields $q_K(0)=\P(T_{-1} > \Delta)$ in 
the setting of \cite[Thm.~1]{DSW18},  where 
$B_\Delta$,  in the notation of the latter, is $\Theta(\Delta)$,
since the 
variance of our increments is uniformly, over $j$ and $N$, 
bounded above and below.  Further, 
\cite[(5),(6),(9)]{DSW18} trivially hold (their Lindeberg
condition 
%\cite[(9)]{DSW18} 
is an easy consequence of the uniform exponential tail of our increments),
% uniform 4-th moment suffices here, even less.
whereby $\liminf_{N} \{ \sqrt{\Delta} \; q_K(0) \} > 0$
(since $U_g(\infty)$ of \cite[Prop.~5]{DSW18} is strictly 
positive in our case, of a constant threshold $g_\Delta:=-1$). 

Returning to our proof, note that by the Gibbs--Markov property 
\cref{def-Gibbs},  the marginal law of $\widetilde \phi_K$ restricted to $I_1' \cup I_2'$,  is obtained by replacing in $\widehat{\Q}_\fh^{(-)}$ the contribution from $\widetilde \phi_K$ on $I_3'$, by the factor $q_K(\widetilde\phi_K(a_K^-))$ and normalizing.  Given our $N^{-c}$ uniform lower bound on the factor $q_K(v)$, it thus suffices to establish \cref{eq:almost-tight-induction} at $r=K$ for 
the modified measure $\widetilde{\Q}_{\fh}^{(-)}$ 
where we completely ignore the path $\widetilde \phi_K$ beyond $a_K^-$.
Proceeding hereafter with $\widetilde{\Q}_{\fh}^{(-)}$,  we fix $b:=\fh \lambda^{K-1}$ and
realize the law of $\widetilde \phi_K$ as the expected 
value of the random laws $\P_{1,I_1'\cup I_2';g,0}^{b,1;0,\star}$
for a single $\frac{b}{N}$-area-tilted random walk on $I_1' \cup I_2'$,  from height $0$ on the left, 
% unspecified height on the right, 
with a zero floor and ceiling $g$ (which is set to be infinite on $I_2'$). 
The expectation here is over a sample from $\widetilde{\Q}_{\fh}^{(-)}$,  of $g = \widetilde \phi_{K-1}$
on $I_1'$.
%and conditional on it, sampling also $v=\widetilde \phi_K(a_K^-)$.
On the \abbr{lhs} of \cref{eq:almost-tight-induction} we have
an increasing event,  so by \cref{lem:Qh-monotonicity}  
we can replace the computation to one with $g \equiv \infty$.
Hence,  we get \cref{eq:almost-tight-induction} at $r=K$, 
out of 
\[
\P_{1,I_1'\cup I_2';\infty,0}^{b,1;0,\star} \Big(\sup_{x \in I_1' \cup I_2'} \phi(x)> N^{1/3+2\delta}\Big) \le C_K N^{-L} \,.
\]
Further,  setting $\mathsf{A}=N^{\delta}$, 
\begin{align}
\P_{1,I_1'\cup I_2';\infty,0}^{b,1;0,\star} \Big(\sup_{x \in I_1' \cup I_2'} \phi(x)& >   N^{1/3+2\delta}\Big) \le 
\P_{1,I_1'\cup I_2';\infty,0}^{b,1;0,\star} \Big( \phi(a_K^- )> \mathsf{A} (N/b)^{1/3}  \Big) \nonumber \\
& +
\sup_{v \le \mathsf{A} (N/b)^{1/3}} \P_{1,I_1'\cup I_2';\infty,0}^{b,1;0,v} \Big(\sup_{x \in I_1' \cup I_2'} \phi(x)> N^{1/3+2\delta}\Big) \,.
\label{eq:bd-A+B}
\end{align}
From \cref{lem:Ofer-A4},  the first term on the right side of \cref{eq:bd-A+B}
is at most $C_K N^{-L}$,  whereas by  \cite[Cor.~5.2]{HKS25} with $\mathsf{R}=\mathsf{A}$
and $I=I'_1 \cup I_2' =\Omega(N)$,
%and $H_b=(N/b)^{1/3}$,  
the rightmost term is also bounded by $C_K N^{-L}$.

Next,  under $\widetilde{\Q}_{\fh}^{(-)}$ we can first sample 
$\widetilde{\phi}_K$ on $[-N,a_K^{-}]$ and then sample the remaining ordered curves 
$\{ \widetilde{\phi}_r,  1 \le r \le K-1 \}$  according to  that measure,  now with 
the specified $\widetilde{\phi}_K(x)$ as their floor.  Having established already 
\cref{eq:almost-tight-induction} at $r=K$,  we thus utilize again 
\cref{lem:Qh-monotonicity} to deduce that it suffices to establish this bound 
for $r \le K-1$ when having instead the floor $h:= N^{1/3+2\delta} \le \widetilde{\phi}_{K-1}(x)$
and left boundary values $\widetilde{\phi}_{r}(-N)=h$.  

Having done so and further reducing $\widetilde{\phi}_r$ by $h$ to have a zero floor,  we arrive at 
precisely the original problem of proving \cref{eq:almost-tight-induction},  just with $K-1$ curves 
(instead of $K$ curves).  The proof is thus concluded by repeating the preceding process 
$K$ times.
\end{proof}

The following tail bound on $y^{(\pm)}_r$
% the argument of $\Q_\star(\cdot)$
 in \cref{eq:w-y-pm} is key to 
our proof of the weak convergence of $\widehat\phi_r$ to $\BB_r$. (See \cref{fig:props-at-pm-delta}, which illustrates this bound vs.\ the one from \cref{prop:N-third-almost-tight}.)
\begin{proposition}\label{prop:v1-pm} 
For any fixed $\epsilon\in(0,\frac1{9})$ and $1 \le r \le K$, 
\begin{align}\label{eq:v1-one-sided-upper}
\lim_{N \to \infty} &
N^{K}\, \sum_{i=1}^{\Delta} \,  \widehat{\Q}_\fh^{(-)}\left(\widetilde\phi_r(a_r^- +i) \leq \E_\star[\widetilde\phi_r(a_r^- +i)] - N^{\epsilon}\sqrt\Delta \right) = 0\,,\\
\label{eq:v1-one-sided-lower}
\lim_{N \to \infty} &
N^{K}\, \sum_{i=1}^{\Delta} \,  \widehat{\Q}_\fh^{(-)}\left(\widetilde\phi_r(a_r^- +i) \geq \E_\star[\widetilde\phi_r(a_r^- +i)] +N^{\epsilon}\sqrt\Delta \right) = 0\,.
\end{align}
\end{proposition}
\begin{remark}\label{rem:v1-pm} 
Considering $y^{(-)}_r=\widetilde{\phi}_r(a_r^-+\Delta)$ of \cref{eq:y-def},  
we get from \cref{eq:v1-one-sided-upper,eq:v1-one-sided-lower} a tail estimate for 
$\|\underline{y}^{(-)}-\E_\star [\underline{y}^{(-)} ]\|_\infty$.
By symmetry $\E_\star [\underline{y}^{(+)}]=\E_\star[\underline {y}^{(-)}]$ 
for $y^{(+)}_r=\widetilde{\phi}_r(a_r^+-\Delta)$,  with the same tail 
bounds for $\|\underline{y}^{(+)}-\E_\star [\underline{y}^{(+)}]\|_\infty$.  
\end{remark}

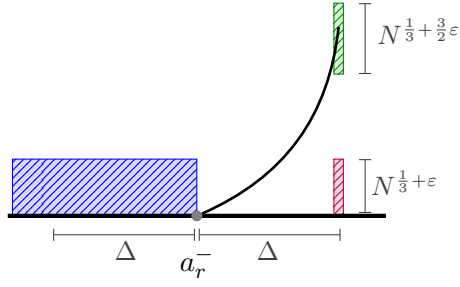
\begin{figure}
    \begin{tikzpicture}
    \begin{scope}[scale=1.25]

    \filldraw[color=blue, pattern=north east lines, preaction={fill=blue!15}, pattern color=blue] (-1.95,0) rectangle (0,0.6);

    \filldraw[color=purple, preaction={fill=purple!15}, pattern=north east lines, pattern color=purple] (1.45,0) rectangle (1.55,0.6);

    \filldraw[color=green!50!black, pattern=north east lines, preaction={fill=green!15}, pattern color=green!50!black] (1.45,1.5) rectangle (1.55,2.25);

    \draw[ultra thick, black] (-2,0)--(2,0);
    \node[circle,scale=0.4,fill=gray,label={[label distance=6pt]below:{$a_r^-$}}] (am3) at (0,0){};

    \draw[|-|,gray!50!black] (0.02,-0.2)-- ++(1.5,-0.0);
    \node[color=gray!50!black,font=\small] at (.75,-0.37) {$\Delta$};
    \draw[|-|,gray!50!black] (-0.02,-0.2)-- ++(-1.5,-0.0);
    \node[color=gray!50!black,font=\small] at (-.75,-0.37) {$\Delta$};

    \draw[|-|,gray!50!black] (1.78,0.03)-- ++(0,0.57);
    \node[color=gray!50!black,font=\small] at (2.2,0.35) {$N^{\frac13+\epsilon}$};

    \draw[|-|,gray!50!black] (1.78,1.5)-- ++(0,0.75);
    \node[color=gray!50!black,font=\small] at (2.35,1.95) {$N^{\frac13+\frac32\epsilon}$};

    \draw[line width=1pt, black] (am3) to[bend right=30] (1.5,2);

    \end{scope}
    \end{tikzpicture}
    \caption{Illustration of the bounds  established in \cref{prop:N-third-almost-tight,prop:v1-pm}. The former states that $\widetilde\phi_r(a_r^--\cdot)$ and 
    $\widetilde\phi_{r+1}(a_r^-+\Delta)$ are each at most $ N^{1/3+\epsilon}$ (the intervals in blue and purple). The latter states that $\widetilde\phi_r(a_r^-+\Delta)$ is concentrated about $\E_\star[\widetilde\phi_r(a_r^-+\Delta)]$ up to  $N^{\frac13+\frac32\epsilon}$ (the green interval).}
    \label{fig:props-at-pm-delta}
\end{figure}

\begin{proof}
We begin with the proof of \cref{eq:v1-one-sided-upper} which concerns the probability of a
decreasing event.  In view of \cref{lem:Qh-monotonicity} it can only increase if we 
minimize $\widetilde\phi_r(a_r^-)=0$  at the left end of $[a_r^-,a_r^-+\Delta]$ and  
remove the (random) floor by forcing $\widetilde{\phi}_{r+1}=-\infty$ throughout that interval.
Having done so,  by standard tail bounds on the moment generating function of the sum of $i \le \Delta$ 
independent---albeit not identically distributed---random variables,  of zero mean,  uniformly 
bounded variances and exponential tails,  
% mgf at theta be exp((C/2) \theta^2 i) for theta < c
% Markov bd gives exp((C/2) \theta^2 i - \theta N^\epsilon \sqrt{\Delta})
% optimum at \theta = N^\epsilon \sqrt{\Delta}/(Ci) when i > (c/C) N^{\epsilon} \sqrt{\Delta} 
% giving then exp(-N^{2 \epsilon} \Delta/(2 C i)),  at smaller i use \theta=c
% to get exp(- c' N^\epsilon \sqrt{\Delta})
the probability that such a sum be under $-N^\epsilon\sqrt\Delta $ 
is at most $\exp(-c N^{2\epsilon})$ (here we used the fact that $N^\epsilon \ll \sqrt\Delta$).

Proceeding to prove \cref{eq:v1-one-sided-lower},  in view of \cref{lem:Ofer-A4} (at $\delta=\epsilon$),  
we may assume that
\[
\widetilde\phi_{r}(a_{r}^-) \le N^{1/3+\epsilon} \,,
\]
whereas by \cref{prop:N-third-almost-tight} (at $\delta=\epsilon/2$),  we further assume \abbr{wlog} that
\[ 
\sup_{ x \in [a_r^-, a_r^-+\Delta]} \{ \widetilde\phi_{r+1}(x) \} \leq N^{1/3+\epsilon} \,.
\]
It then follows from \cref{lem:Qh-monotonicity} that the probability of the increasing event 
in \cref{eq:v1-one-sided-lower} cannot exceed its value when replacing the $(r+1)$-th curve by
a non-random floor $\widetilde\phi_{r}(x) \ge N^{1/3+\epsilon}$ throughout $[a_r^-,a_r^-+\Delta]$,  with 
equality at the left-boundary of this interval (i.e.  $\widetilde\phi_{r}(a_{r}^-) = N^{1/3+\epsilon}$).
With $N^\epsilon \sqrt{\Delta} \gg N^{1/3+\epsilon}$,  the same argument that led to \cref{eq:v1-one-sided-upper}
completes the proof (one merely divides by the $N^{-c}$ lower bound on the probability that the
relevant partial sums are non-negative throughout the interval $[a_r^-,a_r^-+\Delta]$,  which we have 
seen while proving \cref{prop:N-third-almost-tight}). 
\end{proof}

\subsection{Decoupling}\label{sec:decup}

We proceed to establish a local \abbr{clt} which will allow us in the sequel 
to decouple as $N \to \infty$ the laws of $(\Phi^-,\Phi^o,\Phi^+)$,  despite the 
conditioning event of \cref{eq:w-y-pm}.

\begin{lemma}\label{lem:lclt} 
For $\underline{y}^{(o)}$ of \cref{eq:y-def}, let 
\[
f_N(\delta) := \sup \Big\{ \frac{\Q_\star(\underline{y}^{(o)} = \underline{u})}{\Q_\star(\underline{y}^{(o)}=\underline{v})} :
\|\underline{u}\|_\infty \vee \|\underline{v}\|_\infty \le \delta \sqrt{N}, \;\; 
\underline{u},\underline{v} \in \Z^K \Big\} \,. 
\]
Then, 
\begin{equation}\label{eq:fNd}
\lim_{\delta \to 0}\limsup_{N \to \infty} f_N(\delta) = 1 \,.
\end{equation}
\end{lemma}
\begin{proof} Recalling that the \abbr{pmf} 
of $\underline{y}^{(o)}$ under $\Q_\star$ is a product of $K$ terms, 
it suffices to fix $1 \le r \le K$ and establish \cref{eq:fNd} for a single term.  Specifically,  setting
$n := 2 (a_r^+-\Delta)$,  by \cref{def:QI} of the law of $\Phi^o$, 
it suffices to consider \cref{eq:fNd} for the function 
\[
f_{N,r}(\delta) := \sup \Big\{ \frac{q_{0,n}(u)}{q_{0,n}(v)} : u, v \in \Z \cap [-\delta \sqrt{N},\delta \sqrt{N}] \Big\} \,, 
\]
where we have set for $0 \le k \le \ell \le n$ and $w + \E[Z_r(\ell)]-\E [Z_r(k)] \in \Z$ the quantities
\begin{align}\label{def:Sri}
Z_r (k) &:= \sum_{j=a_r^-+\Delta}^{a_r^-+\Delta+k-1} \big[ \tilde{\fX}_r(j) - \sgn(j) \big]\,,   \\
q_{k,\ell}(w)& := \P\Big(Z_r(\ell)-Z_r(k) = w+  \E[Z_r(\ell)]- \E[ Z_r(k)]\Big)\,,
\label{def:q-k-ell}
\end{align}
noting that $\E Z_r(n) = \E Z_r(0) = 0$ by the symmetry in our choice of $\theta_r(\cdot)$.
Turning to bound $f_{N,r}(\delta)$,  recall that $Z_r(k)$ are the partial sums of $k$ independent
$\Z$-valued variables,  where prior to their shift by $\sgn(j)$,  the corresponding 
\abbr{pmf}-s $p_{j \ell} := \P(\tilde{\fX}_r(j)=\ell)$,
are maximal at $\ell=0$ and uniformly bounded away from zero at $\ell=1$ (see \cref{eq:P-theta}). 
Further, recall from \cref{eq:var-bd} that 
$\Var(\tilde{\fX}_r(j))=\Var_{\theta_r(j)}(\tilde{\fX})$ are uniformly bounded above 
and away from zero.  The centered absolute third moments of $\tilde{\fX}_r(j)$ 
are similarly uniformly bounded
(see \cref{eq:P-theta,eq:E-theta,eq:theta_k-def}).  Thus,  \cite[Thm.~VII.4]{Petrov} applies for $Z_r(n)$ and
setting
\begin{align}\label{def:vstar-r}
v_\star^{(r)}(\cdot) &:=v_\star(\cdot)-v_\star(-1+\eta_r)\,, \qquad  \qquad 
\eta_r:=\tfrac{\Delta}{N\alpha_r} \,, \\
% \qquad 1 \le r \le K 
t_k &:= \Var(Z_r(k)) \!\!= \!\! \sum_{j=a_r^-+\Delta}^{a_r^-+\Delta+k-1} 
\Var_{\theta_r(j)} (\tilde{\fX})\,, 
\label{def:BNr}
\end{align}
it follows by comparing \cref{eq:vk-equiv} with \cref{def:vstar-r} that  
$|t_n - N \alpha_r v_\star^{(r)}(1-\eta_r)|$ are uniformly bounded.
Consequently, for some $C=C_r < \infty$ and all $N$, 
\begin{equation}\label{eq:lclt0}
\sup_{w \in \Z}  
\Big| q_{0,n}(w) - \frac{1}{\sqrt{2\pi t_n}}\exp\big(-\frac{w^2}{2 t_n}\big) \Big| \le \frac{C}{N} 
\end{equation}
(cf.~\cite[Formulas (1.1) and (1.9) in Ch. VII]{Petrov}). Since
$N^{-1} t_n \to \alpha_r v_\star(1)>0$ when $N \to \infty$, it thus follows that 
\[
\limsup_{N \to \infty} f_{N,r} (\delta) \le e^{\tfrac{\delta^2}{\alpha_r v_\star(1)}} \,. 
\]
The latter bound 
decays to one as $\delta \downarrow 0$, thereby completing the proof of the lemma.
\end{proof}
\begin{remark}\label{rmk:ext-petrov}  
By the same reasoning as \eqref{eq:lclt0},  for some $C$ and all $n-k$ large enough,
\begin{equation}\label{eq:lcltk}
\sup_{w}  
\Big| q_{k,n} (w) - \frac{1}{\sqrt{2\pi (t_n-t_k)}}e^{-\frac{w^2}{2 (t_n-t_k)}} \Big| \le \frac{C}{t_n-t_k} \,.
\end{equation}
Bounding the \abbr{mgf} of our partial sums yields for some $c<\infty$,  all $\theta$ small and $k \le n$ 
\[
q_{k,n}(w) \le 2 e^{-\theta |w|} e^{c \,  \theta^2 (n-k)} \,,
\]
which with \eqref{eq:lcltk} imply that for some $c_\delta>0$,  some $\kappa_\delta(\rho) \downarrow 0$ 
as $\rho \downarrow 0$ and all $N \ge N_0$, 
\begin{equation*}
%\label{eq:pt-bds}
\max_{k \ge (1-\rho) n} \max_{|w| \ge \delta \sqrt{N}}  \{ q_{k,n}(w) \} \le \frac{\kappa_\delta(\rho)}{\sqrt{N}},  
\qquad \min_{|z| \le \delta \sqrt{N}} \{ q_{0,n}(z) \}  \ge \frac{c_\delta}{\sqrt{N}}\,.
\end{equation*}
Setting
\[
\tau_\rho := \inf\{ k \ge (1-\rho) n :   |Z_r(k) - \E Z_r(k)| > 2 \delta \sqrt{N} \} \,,
\]
we thus deduce that for any $\rho <\rho_0$,  $n \ge n_0$ and $|z| \le \delta \sqrt{N}$,
\begin{align}
\P(\tau_\rho \le n \mid  Z_r(n) =z ) &= \sum_{k=(1-\rho)n}^n \sum_{|w| > 2 \delta \sqrt{N}} \!\!\!
\P(\tau_\rho = k,  Z_r(k) -\E Z_r(k)=w) 
 \frac{q_{k,n}(z-w)}{q_{0,n}(z)}  \nonumber  \\ 
&\le \frac{\kappa_\delta(\rho)}{c_\delta} \P(\tau_\rho \le n) \le  \frac{\kappa_\delta(\rho)}{c_\delta} \,.
\label{eq:tight-k} 
\end{align}
\end{remark}

\Cref{prop:v1-pm} and \cref{lem:lclt} imply the 
asymptotic independence of $\Phi^-$ and $\Phi^+$.
%when $N  \to \infty$.
\begin{lemma}\label{lem:v1-pm}
Fixing $\epsilon \in (0,\tfrac{1}{9})$, 
\begin{align}
\lim_{N \to \infty} \Q_{\fh}\big(\, \|\underline {y}^{(\pm)}- \E_\star [\underline {y}^{(\pm)}] \|_\infty \ge N^{\epsilon} \sqrt{\Delta} \mid A \, \big) &= 0 \,.
\label{eq:v1-pm}
\end{align}
In particular,  setting 
\begin{equation}\label{def:D-delta}
D_{\epsilon} :=  
\big\{ (\Phi^{-},\Phi^{+}) : \|\underline {y}^{(+)}- \underline{y}^{(-)} \|_\infty < 2 N^{\epsilon} \sqrt{\Delta} 
\big\} \,,
\end{equation}
we have that $\Q_{\fh}(D_{\epsilon} \mid A) \to 1$ as $N \to \infty$ and
\begin{equation}\label{eq:p-m-ind}
\limsup_{N \to \infty} 
\| \Q_{\fh}\big((\Phi^-,\Phi^+) \mid A \, \big) - \widehat{\Q}_{\fh}^{(-)} (\Phi^-) 
\widehat{\Q}_{\fh}^{(+)} (\Phi^+) \|_\tv = 0 \,.
\end{equation}
\end{lemma}
\begin{proof} The event $D^c_{\epsilon}$ requires having 
$\max_{\pm} \{ \|\underline {y}^{(\pm)}- \E_\star [\underline {y}^{(\pm)}] \|_\infty \} 
\ge N^{\epsilon} \sqrt{\Delta}$.  Hence,  we deduce from  \cref{prop:v1-pm} and \cref{rem:v1-pm} that 
\begin{equation}\label{eq:whp-D-eps}
\lim_{N \to \infty} \widecheck{\Q}_{\fh}(D_\epsilon) = 1 \,.
\end{equation}
Moreover,  while proving \cref{eq:fNd} we saw that
for some $c=c(\delta)>0$ and any $N$,  
\begin{equation}\label{bd-rdn}
\inf\{ \Q_\star (\underline{y}^{(o)}=\underline{v}) : \underline{v} \in \Z^K, 
\|\underline{v}\|_\infty \le \delta \sqrt{N} \} \ge 2 c N^{-K/2} \,.
\end{equation}
For $\epsilon<\tfrac{1}{9}$,  on the intersection of $\Gamma$ from \cref{eq:w-y-pm} 
with $
%(\Phi^{-},\Phi^+) \in  
D_\epsilon$ one has that 
$\|\underline{y}^{(o)}\|_\infty \le 2 N^\epsilon \sqrt{\Delta} = o(\sqrt{N})$,
allowing us to deduce from the lower bound \eqref{bd-rdn} that for the product measure 
$\widecheck{\Q}_{\fh}(\cdot)$ of \cref{def:QI} and all large enough $N$,
\[
\widecheck{\Q}_{\fh}(\Gamma) \ge  \widecheck{\Q}_{\fh}(\Gamma \cap D_\epsilon)
\ge 2 c N^{-K/2} \widecheck{\Q}_{\fh} (D_\epsilon) 
\ge c N^{-K/2} \,.
\]
In particular,  in view of \cref{def:Gamma},  for any $(\Phi^-,\Phi^o,\Phi^+)$ satisfying 
%for all $N$ large enough, 
$\Gamma$ of \cref{eq:w-y-pm},  
\begin{equation}\label{eq:amir1}
\Q_{\fh} (\Phi \mid A) \le c^{-1} N^{K/2} \widehat{\Q}_{\fh}^{(-)}(\Phi^-) \widehat{\Q}_{\fh}^{(+)}(\Phi^+) \Q_\star (\Phi^o) \,,
\end{equation}
whereby,  with $\Q_\star(\cdot) \le 1$,  the specified decay rates
in \cref{eq:v1-one-sided-upper,eq:v1-one-sided-lower} yield \cref{eq:v1-pm}.  This implies in turn,
by the same reasoning which led to \cref{eq:whp-D-eps},  
that $\Q_{\fh}(D_{\epsilon} \mid A) \to 1$.  Thereby,  from  the elementary inequality
\begin{equation}\label{bd:tv}
\| \P(\cdot\mid B) - \P(\cdot\mid A) \|_\tv \le  \P(B^c\mid A)\,, 
 \qquad \quad \forall B \subseteq A \,,
\end{equation}
we deduce that also 
\begin{equation}\label{eq:amir-tv}
\lim_{N \to \infty} 
\| \Q_{\fh}\big(\, \cdot \mid A \, \big) - \Q_{\fh} \big(\, \cdot \mid A \cap D_\epsilon \, \big) \|_\tv = 0 \,.
\end{equation}
Moreover, with $N^{\epsilon} \sqrt{\Delta} \le \delta \sqrt{N}$ for 
all $N \ge N_0(\delta)$, in view of \cref{def:Gamma}
the Radon--Nikodym derivative of
$\Q_\fh( (\Phi^-,\Phi^+) \mid A )$
\abbr{wrt} the product law
$\widehat{\Q}_{\fh}^{(-)} \otimes \widehat{\Q}_{\fh}^{(+)}$ 
is bounded throughout $D_{\epsilon}$
between $\tfrac{1}{f_N(\delta)}$ and $f_N(\delta)$.  Hence,  by \cref{lem:lclt}, 
\[
\limsup_{N \to \infty} 
\| \Q_{\fh}\big(\, (\Phi^-,\Phi^+)
\mid A \cap D_\epsilon \, \big) - \widehat{\Q}_{\fh}^{(-)}\otimes \widehat{\Q}_{\fh}^{(+)}
\big( (\Phi^-,\Phi^+) \mid D_\epsilon \, \big) \|_\tv = 0 \,.
\]
It then follows, in view of \cref{eq:amir-tv} and the triangle inequality for $\|\cdot\|_{\tv}$,  that 
\[
\limsup_{N \to \infty} 
\| \Q_{\fh}\big(\, (\Phi^-,\Phi^+)
\mid A \, \big) - \widehat{\Q}_{\fh}^{(-)}\otimes \widehat{\Q}_{\fh}^{(+)}
\big( (\Phi^-,\Phi^+) \mid D_\epsilon \, \big) \|_\tv = 0 
\]
from which we deduce,  by \cref{eq:whp-D-eps,bd:tv},   
that \cref{eq:p-m-ind} also holds.
\end{proof}

\begin{proof}[Proof of \cref{lem:non-crossing-BB}]
Recall from \cref{lem:tilting} that $\mu_{N,K}(\cdot)$ of \cref{eq:1d-sos-many-curves-phi} is 
%for $\underline{\widetilde \phi}$, 
$\Q_\fh (\cdot \mid B)$ 
%of \cref{def:Qh}, 
for the event $B$ of \cref{def:B}.  Thus,  by \cref{bd:tv}
it suffices to show that for the event $A$ of \cref{def:A},  
\begin{equation}\label{eq:nc}
 \lim_{N \to \infty} \Q_\fh (B^c \mid A) = 0 \,.  
\end{equation}
To this end,  setting $\widetilde{\phi}_{K+1}(\cdot) \equiv \phi^\star_{K+1}(\cdot) \equiv 0$, 
recall that $a_r^-=-a_r^+$ and 
\[
B^c \cap A=\bigcup_{r} \bigcup_{|i| < a_r^+ - \Delta}  \{ \widetilde{\phi}_{r} (i)
< \widetilde{\phi}_{r+1} (i) \} \cap A  \,.
\]
Utilizing the union bound and 
the invariance to $\sgn(i)$ of $\Q_{\fh}(\widetilde{\phi}_{r} (i) < \widetilde{\phi}_{r+1} (i) \mid A)$, 
we thus deduce from \cref{eq:amir1} that 
\begin{align}\label{eq:bd-pre-partial-sum}
\Q_\fh(B^c \mid A) &\le 2 K c^{-1} N^{1+K/2} \max_{r,  \Delta< i \le a_r^+}
\Q_{\fh}^I 
 \Big( \widetilde{\phi}_{r} (a_r^- +i) < \widetilde{\phi}_{r+1} (a_r^- +i) \Big) \,,
\end{align}
for independent partial sums 
$\{\widetilde{\phi}_{r+b}(a_{r+b}^-+i)-\widetilde{\phi}_{r+b}(a_{r+b}^-+\Delta),  i > \Delta\}$
from $\Q_\star$,  with $b=0,1$,  independently of $\widetilde{\phi}_r(a_r^-+\Delta)$ and 
$(\widetilde{\phi}_{r+1}(i),  i \le a_{r+1}^-+\Delta)$ whose joint law is induced by 
$\widehat{\Q}_{\fh}^{(-)}$.  
We express the events in \cref{eq:bd-pre-partial-sum} via  the zero-mean partial sums
\begin{equation}\label{eq:Zri-o}
\widehat{Z}_{r}(i):= \widetilde{\phi}_{r}(i) - \E_\star \widetilde{\phi}_{r} (i) \,.
\end{equation}
More precisely,  thanks to \cref{eq:v1-one-sided-upper} we replace $\widehat{Z}_r(a_r^-+\Delta)$ 
% $\widetilde\phi_r(a_r^- +\Delta) - \E_\star [ \widetilde{\phi}_r(a_r^- +\Delta)]$
by $- N^{\epsilon}\sqrt\Delta$ and using
\begin{equation}\label{def:bar-Zri}
\bar Z_{r}(i):= \widehat{Z}_{r}(a_r^- + \Delta) - \widehat{Z}_r(a_r^-+i) \,,
%\E_\star [\widetilde{\phi}_{r} (a_r^-+i)]  -  \widetilde{\phi}_{r}(a_r^-+i) 
% +  \widetilde{\phi}_r(a_r^-+\Delta) - \E_\star[ \widetilde{\phi}_r(a_r^-+\Delta) ] 
\end{equation}
we deduce that it suffices to show that for all $\ell \in (\Delta,a_r^+]$, 
\begin{align}\label{eq:bd-amir0}
&  \widehat{\Q}_{\fh}^{(-)} \times \Q_\star (\bar Z_r(\ell) > V_{r}(\ell) ) \le N^{-(K+1)} \,, \\
\hbox{where} \quad 
& V_r(\ell) :=  \E_\star [\widetilde{\phi}_r(a_r^-+\ell)]  
% + \big( \widetilde{\phi}_r(a_r^-+\Delta) - \E_\star[\widetilde{\phi}_r(a_r^-+\Delta)] \big) 
- N^\epsilon \sqrt{\Delta}
- \widetilde{\phi}_{r+1}(a_r^-+\ell) \,.
\label{def:Vri}
\end{align}
To this end,   recall that $|\theta_k(j)| \le \theta_\star < \beta$ 
for any $k \le K$ and all $|j| \le N$ (cf.~\cref{eq:theta_k-def}),  hence it follows from  
\cref{eq:tilde-phi-def} that for some $c<\infty$,
\[
\log \E_\star\Big[e^{s \bar Z_{r}(\ell)}\Big]
= \sum_{j=a_r^-+\Delta}^{a_r^-+\ell-1} \Big(
\Lambda(\theta_r(j)-s) - \Lambda(\theta_r(j)) + s \Lambda'(\theta_r(j)) \Big)
\le c (\ell-\Delta) s^2 \quad \forall |s| \le \frac{1}{c} \,.
\]
That is,  $\bar Z_r(\ell)$ is in the collection $\Gamma(2 c (\ell-\Delta),c)$ of sub-Gamma 
variables of variance factor $2 c (\ell-\Delta)$ and scale parameter $c$,  yielding the tail estimate
\begin{equation}\label{eq:sub-Gamma}
\Q_\star(|\bar Z_r(\ell)| \ge 2 \sqrt{c (\ell-\Delta) t} + ct) \le 2 e^{-t} \,, \qquad \forall t >0,  \ell > \Delta 
\end{equation}
(see \cite[Eq.~(2.6)]{BLM13}).  Utilizing \cref{eq:sub-Gamma} we now show 
\cref{eq:bd-amir0} in three sub-intervals of $(\Delta,a_r^+]$.

\noindent{\bf I.} Having $\ell \in (\Delta,a_r^+-a_{r+1}^+]$.  For some $\eta \in (0,1]$ and all $\ell \in (\Delta, a_r^+]$, 
\begin{align*}
 \E_\star [\widetilde{\phi}_r(a_r^-+\ell)] - N^{\epsilon} \sqrt{\Delta}
& \ge 3 \eta \ell^2 N^{-1} - N^{\epsilon} \sqrt{\Delta} 
%\\ &
 \ge 2 \eta \ell^2 N^{-1}
 \,.
\end{align*}
With $a_r^+-a_{r+1}^+ = a_{r+1}^--a_r^-$,  we can also assume,  by  \cref{prop:N-third-almost-tight} 
at $\delta=\epsilon/2$,  that in this range,  
\[
V_r(\ell) \ge  2 \eta \ell^2 N^{-1} - \widetilde\phi_{r+1}(a_r^-+\ell) \ge 2 \eta \ell^2 N^{-1}
- K N^{1/3+\epsilon}  \ge \eta \ell^2 N^{-1} 
\]
and \cref{eq:bd-amir0} follows,  since by \cref{eq:sub-Gamma} at $t_\ell:=\frac{\eta^2}{9c} \ell^3/N^{2}$ we have that for some $c'>0$, 
\[
\Q_\star \big( \bar Z_r(\ell)  >  \eta \ell^2 \, N^{-1}  \big) \le 2 e^{-t_\ell} \le 2 e^{-t_\Delta} \le e^{-c' N^{2 \epsilon}} 
 \qquad  \forall \ell \in (\Delta,a_r^+] \,.
\]
{\bf II.} Having $\ell = i + a_r^+- a_{r+1}^+$ with $i \in (\Delta,a_{r+1}^+]$.  By
\cref{eq:v1-one-sided-lower},  we have that \abbr{whp} 
\[
\widehat{Z}_{r+1}(a_{r+1}^- +\Delta) \le N^{\epsilon}\sqrt\Delta \,,
\]
allowing us to replace $V_r(\ell) - \bar Z_{r+1}(i)$,  by the non-random
\begin{equation}\label{def:Dri}
D_{r}(i) := \E_\star [\widetilde{\phi}_r(a_{r+1}^-+i) - \widetilde{\phi}_{r+1}(a_{r+1}^-+i)] - 2 N^\epsilon \sqrt{\Delta} \,.
\end{equation}
In view of \cref{eq:phi-star-equiv},  for some $\eta'>0$ and all $N$ large enough
\begin{equation}\label{eq:bd-Dri}
\inf_{i \in [0,a_{r+1}^+]} \{ D_r(i) \} \ge \frac{N}{2} \inf_{t \in (-\alpha_{r+1},0]} \{ 
\phi_r^\star(t) - \phi_{r+1}^\star(t) \} \ge 2 \eta' N 
\end{equation}
and \cref{eq:bd-amir0} holds since by \cref{eq:sub-Gamma},  for some $c>0$ and all $N$,
\[
\Q_\star (\bar Z_r(\ell)> \eta' N) + \Q_\star( \bar Z_{r+1}(i) < - \eta' N) \le 4 e^{-c \sqrt{N}} \,.
\]
{\bf III.} Having $\ell = i + a_r^+- a_{r+1}^+$ with $i \in (0,\Delta]$.  In view of  \cref{prop:N-third-almost-tight} 
we may assume that both $\widetilde{\phi}_{r+1}(a_{r+1}^-)$ and the curve 
$\widetilde{\phi}_{r+2}(\cdot)$ on $I:=[a_{r+1}^-,a_{r+1}^-+\Delta]$ do not exceed 
the non-random floor $h :=  N^\epsilon \sqrt{\Delta}$.  With the \abbr{lhs} of \cref{eq:bd-amir0} 
increasing in terms of the  $(r+1)$-th curve from $\widehat{\Q}_{\fh}^{(-)}$
restricted to $I$,  by the Gibbs property of that system and 
% the monotonicity 
\cref{lem:Qh-monotonicity},  the probability in  \cref{eq:bd-amir0}
is at most the one obtained when drawing $\widetilde{\phi}_{r+1}(\cdot)-h$ in \cref{def:Vri} 
via the partial sums of $\Q_\star$ conditioned to stay non-negative throughout $I$.  Recall,  as in 
our proof of \cref{prop:N-third-almost-tight},  that such conditioning event occurs with 
probability $\Omega(\Delta^{-1/2})$.  Hence,  it suffices to show that
\begin{equation}\label{eq:bd-amir9}
\Q_\star^{\otimes 2} (\bar Z_r(\ell) > V_{r}(\ell) - h ) \le N^{-(K+2)} \,.
\end{equation}
Next,  
it follows from \cref{eq:bd-Dri} and \cref{eq:sub-Gamma} that 
\begin{align*}
\Q_\star^{\otimes 2} (\bar Z_r(\ell) +  \widehat{Z}_{r+1}(a_{r+1}^-+i) > D_r(i) ) 
& \le  \Q_\star (\bar Z_r(\ell)> \eta' N) + \Q_\star(  \widehat{Z}_{r+1}(a_{r+1}^-+i)  > \eta' N) \\
&  \le 4 e^{-c \sqrt{N}} \,.
\end{align*}
Analogously to case {\bf II},  here $V_{r}(\ell) + \widehat{Z}_{r+1}(a_{r+1}^-+i) - h  = D_r(i)$ of \cref{def:Dri}, 
yielding \cref{eq:bd-amir9} and thereby completing our proof.
\end{proof}

\section{Proof of \cref{thm:1}}\label{sec:pf}

In view of \cref{lem:non-crossing-BB} and \cref{lem:Q-A} we choose $\Phi$ throughout 
this section according to $\Q_{\fh}(\Phi \mid A)=\widecheck{\Q}_{\fh}(\Phi \mid \Gamma)$ and fix 
$\epsilon>0$ small,  whereby $\epsilon':=\frac{1}{6}-2 \epsilon > 0$.  
\subsection{The convergence to Brownian bridges}\label{subsec:X0}
%With $\E_\star[\widetilde{\phi}_r(i)]=0$ outside $[a_r^-,a_r^+]$,  
 Recall  \cref{eq:amir1} that $\widecheck{\Q}_{\fh}(\Phi \mid \Gamma)$ is at most $O(N^{K/2})$ times the 
un-conditional product measure $\widecheck{\Q}_{\fh}(\Phi)$ of \cref{def:QI}.  
Also,  our offset $N \phi^\star_r(\cdot)$ is zero outside $[-\alpha_r,\alpha_r]$. 
Hence,  the (un-conditional) tail bounds 
of \cref{eq:almost-tight-induction} and \cref{eq:v1-one-sided-lower} imply that for any $\delta>0$, 
% \cref{prop:N-third-almost-tight} 
\[
\lim_{N \to \infty} 
\widecheck{\Q}_{\fh}\Big( \max_{r,  \alpha_r(1-\eta_r) \le |t| \le 1} \{ |\widehat{\phi}_r(t)| \} \ge \delta \mid \Gamma\Big) 
= 0
\]
(where $\eta_r=\frac{\Delta}{N \alpha_r}$,  see \cref{def:vstar-r}).
It thus suffices to consider in \cref{thm:1}, \cref{part:a} only the weak 
convergence for the part of $\underline{\widehat{\phi}}:[-1,1] \to \R^K$ induced by $\Phi^o$ of \cref{dfn:Phis}.
Specifically,  prior to centering and scaling,  those paths are represented in terms of the partial sums $Z_r(k)$ 
of \cref{def:Sri} as
\begin{equation}\label{eq:phi-o}
\widetilde{\phi}_r(a_r^-+\Delta+k)= y_r^{(-)}+Z_r(k),   \qquad 0 \le k \le n := 2(a_r^+-\Delta)\,, 
\end{equation}
conditional upon $Z_r(n)=y_r^{(+)}-y_r^{(-)}$.  While the law of $y_r^{(\pm)}$ of \eqref{eq:y-def}
is induced by $\widehat{\Q}^{(\pm)}_{\fh}(\cdot)$,  setting 
$\underline{z}^{(\pm)}:=\underline{y}^{(\pm)} - \E_\star [\underline{y}^{(\pm)}]$ we know from 
\cref{lem:v1-pm} that it suffices to consider only the event $D_\epsilon$,  where
\[
\|\underline {z}^{(\pm)}\|_\infty \le 2 N^{\epsilon} \sqrt{\Delta}  \le \frac{1}{2} N^{1/2-\epsilon'} \,.
\] 
In particular,  given our $\sqrt{N}$ scaling,  we are thus allowed to shift the partial sums, 
replacing $y_r^{(-)}$ in \eqref{eq:phi-o} by $\E_\star[y^{(-)}_r]=\E_\star[y^{(+)}_r]$
%\begin{equation}\label{eq:phi-o}
%\widecheck{\phi}^o_r(i) - \widecheck{\phi}^o_r(a_r^-+\Delta)
%+ z^{(-)}_r  + \E_\star [  \widecheck{\phi}^o_r(a_r^-+\Delta) ] \,,  \quad 
%i \in [a_r^-+\Delta,a_r^+-\Delta]\,,
%\end{equation}
and eliminate the effect of $\widehat{\Q}^{(\pm)}_{\fh}(\cdot)$ by establishing the
relevant weak convergence conditioned on the specified values 
\[
Z_r(n) = z,  \quad \hbox{uniformly} \;\;  \hbox{over} \;\; |z| \le N^{1/2-\epsilon'}\,.
\]
Further,  in view of \eqref{eq:phi-star-equiv} we replace the offset term $N \phi_r^\star(t'_k)$
at $t'_k=(a_r^-+\Delta+k)/N$ by 
\[
\E [Z_r(k)] + \sum_{j=a_r^-}^{a_r^-+\Delta-1} \big[\E_{\theta_r(j)}(\tilde\fX) - \sgn(j)\big] 
= \E[Z_r(k)] + \E_\star[y_r^{(-)}] \,.
\]
Doing so reduces the proof of \cref{thm:1}, \cref{part:a} to showing that for fixed $1 \le r \le K$, 
the linearly interpolated,  re-scaled partial sums
\begin{equation}\label{def:hatSri}
S_N^\dagger(t'_k) := \frac{1}{\sqrt{N \alpha_r}} (Z_r(k) - \E [Z_r(k)]),  \qquad 0 \le k \le n \,,
\end{equation}
conditioned on the value $z$ of $Z_r(n)$,
converge weakly to $\BB_r(\cdot)$ of \eqref{def:un-X}.  To this end,  with 
$v_\star:[-1,1] \to [0,\frac{2}{\theta_\star}]$ of  \cref{eq:V-def} continuous and increasing,
it suffices to prove that the time-changed process 
\begin{equation}\label{def:Sdd}
S_N^\ddagger(s) := S_N^\dagger(\alpha_r v_\star^{-1}(s))\,,
\end{equation}
converges weakly on $C([0,v_\star(1)],\R)$,  to the Brownian bridge on $[0,v_\star(1)]$,   
\begin{equation}\label{dfn:BB-alt}
\BB_\star(s) := \BB_r(\alpha_r v_\star^{-1} (s)) \,.
\end{equation}
Further,  $\max_{N} |N\alpha_r + a_r^-|$ is finite (see \cref{def:akpm}),  while we have for $t_k$ of 
\eqref{def:BNr} and $v_\star^{(r)}(\cdot)$ of \cref{def:vstar-r}, that
\[
\max_{N,k} \big| N \alpha_r v^{(r)}_\star\big(t'_k/\alpha_r
%-1+\eta_r+\tfrac{k}{N \alpha_r}
\big) -t_k \big| < \infty \quad  \text{and} \quad \frac{t_n}{N \alpha_r} \to v_\star(1)
\]
(see \cref{eq:V-def-as-int,eq:vk-equiv}).  Consequently,  we  may and will instead prove the 
weak convergence to $\BB_\star(\cdot)$ of $S_N(\cdot)$ obtained by linear interpolation 
between the discrete samples
\begin{equation}\label{def:SN}
S_N(v_\star(1) t_k/t_n) = S_N^\ddagger (v_\star(t'_k/\alpha_r)) = S_N^{\dagger}(t'_k),   \qquad 0 \le k \le n 
\end{equation}
(since
% as $\eta_r \to 0$,  
the covariance of
$\BB_\star\Big(v_\star \circ (v_\star^{(r)})^{-1} \big(\frac{v_\star^{(r)}(1-\eta_r)}{v_\star(1)} s\big) \Big)$ converges to that of $\BB_\star(s)$).  Without conditioning on $S_N(v_\star(1))$,  such weak convergence is an immediate consequence of the invariance principle (see, e.g.,  \cite{Sakhanenko11}).  The conditional version,
in the \abbr{iid} case,  goes back to \cite{Liggett68}.  However,  we did not find a proof of its non-\abbr{iid} extension,  
so provide such here.   Specifically,  \cref{eq:tight-k} amounts to uniform tightness of $S_N(\cdot)$ at $v_\star(1)$,  
given $S_N(v_\star(1))=\frac{z}{\sqrt{N \alpha_r}}$,  thanks to which the claimed weak convergence
follows from showing that for any $\rho>0$ small enough and $F \in C_b(C([0,v_\star(1-3\rho)],\R),\R)$, 
\begin{equation}\label{eq:F-conv}
\lim_{N \to \infty} \sup_{|z| \le N^{1/2-\epsilon'}}\Big\{ \Big| \E [F(S_N(\cdot)) \mid Z_r(n)=z] 
- \E [F(\BB_\star(\cdot))] \Big| \Big\} = 0 \,.
\end{equation}
Proceeding to show \eqref{eq:F-conv} we set the short-hand $\bar n:=(1-\rho) n$ and 
$\bar t :=v_\star(1) t_{\bar n}/t_n$,
% be the time-change bar corresponding to $t'_{\bar n}$, 
noting that 
by \eqref{eq:tight-k} we may replace $F(S_N(\cdot))$ by $F(S_N(\cdot)) \one_{\{|S_N(\bar t)| \le 1\}}$.  Further,  with
$\bar t \ge v_\star(1-3\rho)$ for all large $N$,  we get in terms of 
$q_{\bar n,n}(\cdot)$ from \cref{def:q-k-ell} that 
% originally it was $v_\star((1-2\rho)(1-\eta_r)) - v_\star(-1+\eta_r) \to v_\star(1-2\rho)$ as $\eta_r \to 0$ with $N$
% and the extra factor $v_\star(1)/(v_\star(1-\eta_r)-v_\star(-1+\eta_r)) \to 1$ with $N$ to keep the above property. 
\[
\E [F(S_N(\cdot)) \one_{\{|S_N(\bar t)| \le 1\}} \one_{\{Z_r(n)=z\}}] = \sum_{|w| \le 1}
q_{\bar n,n}(z-\sqrt{N \alpha_r} \,  w)  \E \big[F(S_N(\cdot)) \one_{\{S_N(\bar t)=w\}} \big] \,,
\]
which by \eqref{eq:lcltk} is within $\frac{C_\rho}{N}$ of 
\[
\frac{1}{\sqrt{2\pi (t_n-t_{\bar n})}} \E \Big[ e^{-\frac{(\sqrt{N \alpha_r} S_N(\bar t)-z)^2}{2 (t_{n}\,-t_{\bar n})}} 
F(S_N(\cdot)) \one_{\{|S_N(\bar t)| \le 1\}}  \Big] \,.
\]
The same reasoning,  with $F(\cdot)=1$,  yields such a formula for $\P(Z_r(n)=z) \ge \frac{c_\delta}{\sqrt{N}}$
(as $|z| \le \delta \sqrt{N}$),  so taking the ratio of the two expressions results in 
\begin{equation}\label{eq:F-cond}
\E [F(S_N(\cdot)) \mid Z_r(n)=z] = 
\frac{\E \Big[ e^{-\frac{S_N(\bar t)^2}{2 \sigma_N^2}} F(S_N(\cdot))  \Big]}{\E \Big[ e^{-\frac{S_N(\bar t)^2}{2 \sigma_N^2}} \Big]}
+O(N^{-1/2}) + O(\kappa_\delta(\rho)) \,,
\end{equation}
where $\bar t \to v_\star(1-\rho)$,  hence
$\sigma_N^2 := \frac{1}{N \alpha_r} (t_n-t_{\bar n}) \to v_\star(1)-v_\star(1-\rho)$.

Now,  recall that the zero-mean increments $\zeta_j$ of $Z_r(\cdot) - \E[Z_r(\cdot)]$ satisfy 
condition \cite[(3.2)]{Sakhanenko11} (denoted there $\xi_j$).  Namely,   for some $\lambda_0>0$ and $C<\infty$,
\[
\lambda_0 \E \big[ |\zeta_j|^3 e^{\lambda_0 |\zeta_j|} \big] \leq 6 \Var(\zeta_j),  \qquad
\sum_{j=0}^{n-1} \Var(\zeta_j) \le C N \,.
\]
Thus,  following  \cite[Lemma 3.1]{Sakhanenko11} with a Brownian scaling by $t_n/v_\star(1)$,  and recalling 
\cref{def:SN},
yields
a coupling $\P^\star$ of a standard Brownian motion $W(\cdot)$ and $S_N(\cdot)$ on $[0,\bar t]$
(see \cite[(2.2)]{Sakhanenko11}), such that for some $\delta_N \to 0$, 
%for some finite $C_\ell$ any $\ell$ and all $N$, 
as $N \to \infty$,  
\begin{equation}\label{eq:KMT-tail}
\lim_{N \to \infty} \P^\star \Big( 
\sup_{s \in [0, \bar t]} \{ | S_N(s) -W(s) | \} \ge \delta_N \Big) = 0
%\ge C_\ell \frac{\log N}{\sqrt{N}} \Big) \le N^{-\ell} 
% =  \frac{C N}{x^2} e^{-\lambda_0 x} 
\end{equation}
(here \cite[(3.3)]{Sakhanenko11} holds at $\widetilde{\beta}=0$,  so we get \cref{eq:KMT-tail}
by considering \cite[(3.4)]{Sakhanenko11} for $z=0$ and $x \ge C \log N$ for $C$ large enough).
%=\frac{C_\ell}{C_0} \sqrt{\alpha_r} \log N$
%As before,  re-parametrizing $v_\star(1)t_k/t_n$
%in \eqref{def:hatSri} still supplies \eqref{eq:KMT-tail}. 
% when $N \to \infty$.  
This allows us when taking $N \to \infty$ to replace on the \abbr{rhs} of 
\eqref{eq:F-cond} the process $S_N(\cdot)$ by $W(\cdot)$ and the parameters $\bar t$
and $\sigma_N$ by their limits,  to deduce that uniformly in our range of $z$,  
\[
\E [F(S_N(\cdot)) \mid Z_r(n)=z] = \frac{\E \Big[ e^{-\frac{W(v_\star(1-\rho))^2}{2 (v_\star(1)-v_\star(1-\rho))}} 
F(W(\cdot))  \Big]}{\E \Big[ e^{-\frac{W(v_\star(1-\rho))^2}{2 (v_\star(1)-v_\star(1-\rho))} } \Big]}
+o_N(1) + O(\kappa_\delta(\rho)) \,.
\]
It is easy to verify that the first term on the \abbr{rhs} is precisely $\E[F(\BB_\star(\cdot))]$,
so taking first $N \to \infty$ and then $\rho \downarrow 0$ yields \cref{eq:F-conv}.  As explained before,
having \cref{eq:F-conv} for all $F \in C_b(C([0,v_\star(1-3\rho)],\R),\R)$ and $\rho$ small enough, implies the stated
weak convergence on $C([0,v_\star(1)],\R)$ and thus concludes the proof of \cref{thm:1}, \cref{part:a}.

\begin{remark}
\label{rem-Senya}
Our proof in fact yields the following, via a comparison with an (inhomogeneous) random walk bridge whose starting and end points are 
bounded by $O(N^{1/3+\epsilon})$ with arbitrary $\epsilon>0$: fixing $\delta\in (0,1/3)$,  the random variable 
$N^{\delta/2} \widehat\phi_k(-\alpha_k+N^{-\delta})$
converges in distribution to a zero mean Gaussian random variable of variance $v_\star'(-1)/\alpha_k$. 
In particular, this resolves in the affirmative \cite[Conjecture 2]{OSV24}, which deals with this question for $K=1$.
\end{remark}

\subsection{The convergence to the Brownian \texorpdfstring{\abbr{gate}}{GATE}}
In view of \cref{eq:p-m-ind} it suffices to prove the joint convergence of $\{\underline{\psi}^{(-r)}_N,  1 \le r \le K\}$,
to an independent Brownian \abbr{gate} of law $\mu^{\mathfrak{o}}_{\lambda,\fa_r,K+1-r}$,  subject to $I_N^{(-r)}$-shifts,
under the law induced by the \emph{unconditional} $\widehat{\Q}_{\fh}^{(-)}$
(the corresponding limit then must hold also for $\{\underline{\psi}^{(r)}_N,  1 \le r \le K\}$,  
with the two limits necessarily mutually independent).  To this end,  
we proceed in two steps: first,  the following proposition 
establishes the tightness of  $N^{-1/3} \widetilde\phi_r (a_{r-1}^-+2 \Delta)$
and of $N^{-1/3} \widetilde\phi_r (a_{r}^--2 \Delta)$.  This allows us in a 
second step to get the stated convergence by applying results from \cite{Serio23}.
% (\widehat v_2^{(-(r-1))},v_1^{(-r)})$.
\begin{proposition}\label{lem:N-third-tight} 
For any $\eta>0$ there exists $M_\eta<\infty$ such that 
\begin{equation}
    \label{eq:endpoints-N-13-top}
\widehat{\Q}_{\fh}^{(-)} \left( \widetilde\phi_r (a_{r-1}^-+2 \Delta) \;\vee\;
\widetilde\phi_r (a_{r}^-- 2\Delta) > M_\eta N^{1/3} \right) \le \eta \,, \qquad \forall 1 \le r \le K \,.
\end{equation}
\end{proposition}
\begin{proof} We will show below that for some $c_0 \ge 3$,  any $a \in [a_{r-1}^- +  2 \Delta,a_r^- - \Delta]$, 
$m \in [c_0, N^{\epsilon/3}]$,  $m'=(\log m)^2$ and 
$\mathsf{w}_m := \lceil 2 K(1+\fh^{-2}) m^2 N^{2/3} \rceil $,
\begin{equation}\label{eq:endpoints-N-13}
{\widehat\Q}_\fh^{(-)}\Big(\cE_{J'}(m') \;\Big|\; \cE_{J}(m) \Big) \ge 1 - m^{-10}\,,
\end{equation}
where $J':= J_r (m)$ and $J:=J_{K+1}(m)$ for $J_\ell(m) := [a - \ell  \mathsf{w}_m ,  a + \ell \mathsf{w}_m]$,
while
\begin{equation}\label{def:EJm}
\cE_{J}(m) := \Big\{
\sup_{x \in J} \{ \widetilde\phi_{r}(x) \} \le (K+1-r) m  N^{1/3} \Big\}\,.
\end{equation}
Then,  fixing $\eta \in (0,1)$ we apply \cref{eq:endpoints-N-13}  
at $a=a_{r-1}^-+2 \Delta$ and $a=a_r^--2\Delta$,  for $m_j=(\log m_{j-1})^2$
starting with $m_0=N^{\epsilon/3}$.  Indeed,  $(K+1) \mathsf{w}_{m_0} \le \Delta$ for large $N$,  with 
$J_{K+1}(m_0) \subseteq [-N, a_r^-]$ and $\widehat{\Q}_\fh^{(-)}(\cE_{J}(m_0)) \ge 1 - m_0^{-10}$ 
thanks to \cref{prop:N-third-almost-tight} (at $\delta=\frac{1}{6}\epsilon$).  Further,  taking
$(K+1) (\log c_0)^4 \le c_0^2$ leads to $J_{K+1}(m_{j}) \subset J_r(m_{j-1})$ so we can 
iterate down until $m' \le c_0 / \eta$ with non-empty nested intervals,  ending 
at $\widetilde\phi_r(a) \le K (c_0/\eta) N^{1/3}$,  where by the union bound,  the 
exceptional probability is at most 
\[
\sum_{j \ge 0} m_j^{-10} \le \frac{\eta}{2} \,, 
\]
thereby yielding \cref{eq:endpoints-N-13-top}.

Turning to the proof of \cref{eq:endpoints-N-13},  we fix $m \in [c_0,N^{\epsilon/3}]$, 
setting $\cE_{\ell,J}(m)$ as in \cref{def:EJm} except for replacing there $r$ by $\ell$ and 
show that 
\begin{equation}\label{eq:endpoints-N-13-induction}
{\widehat\Q}_\fh^{(-)}\left((\cE_{\ell,J_{\ell}}(m'))^c \mid \cE_{\ell+1,J_{\ell+1}} (m'),\, \cE_{J_{K+1}}(m) \right) 
\le \frac{1}{K} m^{-10} \,,  \qquad \forall r \le \ell \le K \,.
\end{equation}
This yields \cref{eq:endpoints-N-13} as
% $\widetilde\phi_{K+1}\equiv 0$,  hence 
$\cE_{K+1,[-N,N]}(0)$ always holds and   
\cref{eq:endpoints-N-13} would follow by decomposing $(\cE_{J'}(m'))^c=(\cE_{r,J_r}(m'))^c$
according to the largest $\ell \le K$ for which $(\cE_{\ell,J_\ell}(m'))^c$ occurs.  
Proceeding to show \cref{eq:endpoints-N-13-induction},  
clearly $J_{K+1}(m) \subset [a_{r-1}^-+\Delta, a_r^-]$ for such values of $a$,  $m$ and large enough $N$.
Thus,  restricting $\widehat{\Q}_{\fh}^{(-)}$  
to any interval $J \subset J_{K+1}(m)$ with specified boundary values $\underline{u},\underline{v}$ yields the law 
$\P_{K+1-r,J;g,0}^{{\fh} \lambda^{r-1},\lambda;\underline{u},\underline{v}}$
of 
a \abbr{rw} \abbr{gate} 
with centered 
increments
% (of law $\tilde{\fX}+1$ on $\Z$ under $\P_{-\theta_\star}$),  
of finite exponential moments, 
% and variance $\Lambda''(\theta_\star)$,  
where, thanks to $\cE_{J_{K+1}}(m)$,  
 its \abbr{rw} bridges have 
a ceiling of
\[ g \equiv (K+1-r) m N^{1/3} 
\,.\]
The event $(\cE_{\ell,J_\ell}(m'))^c$  is increasing in terms of 
% the curve
$\{ \widetilde{\phi}_{\ell}(x) : x \in J_\ell\}$,  which under the event  
$\cE_{\ell+1,J_{\ell+1}}(m')$ has a (random) floor not exceeding 
$h\equiv (K-\ell) m' N^{1/3}$
(as $J_\ell \subset J_{\ell+1}$).  It thus follows from 
\cref{lem:Qh-monotonicity} that, for $b=\fh \lambda^{\ell-1} >0$,  the \abbr{lhs} 
of \cref{eq:endpoints-N-13-induction} is at most the following probability for a single area-tilted curve $\phi$:
\begin{equation}
\P_{1,J_{\ell+1};\infty,0}^{b,1;g,g} \Big(\sup_{x \in J_\ell}\{  \phi(x) \} > m' N^{1/3}  \Big) \,.\label{eq:single-curve-prob-reduction}
\end{equation}
To bound the latter probability,  
we write $J_{\ell+1}\setminus J_\ell = I' \cup I''$,  where 
\begin{align*} I' =\left[a- (\ell+1) \mathsf{w}_m, a - \ell \mathsf{w}_m\right] \, ,
%\qquad \hbox{and} 
\qquad 
I'' = \left[a + \ell \mathsf{w}_m, a + (\ell+1) \mathsf{w}_m\right]\,,
\end{align*}
and appeal to \cite[Lem.~4.2]{HKS25} to conclude that with high probability there exist 
some $x'\in I'$ and $x''\in I''$ so that $\phi(x') \vee \phi(x'') =O(N^{1/3})$.  

\noindent
The requirements of that lemma are as follows:
\begin{itemize}
\item  The boundary conditions should satisfy $u\vee v \le N/b$; in our application, we have $g \asymp m N^{1/3} \leq K N^{1/3+\epsilon/3}$,  so this holds true with room to spare. 
\item Having interval length of at least $2 (\sqrt{u}\vee\sqrt{v})\sqrt{N/b}$; here $|J_{\ell+1}| = 2(\ell+1)\mathsf{w}_m$,  while $2 \sqrt{g}\sqrt{N/b} \leq 2 \sqrt{K/\fh} \sqrt{m} N^{2/3}$, so this holds true by our choice of $\mathsf{w}_m$.
%(we could even put $\sqrt m$ instead of $m^2$ there).
\item The length of the sub-interval where we apply the lemma should be of order $|J_{\ell+1}|$; here 
$\frac{|I'|}{|J_{\ell+1}|} = \frac{|I''|}{|J_{\ell+1}|} = \frac{1}{2(\ell+1)}$,  uniformly bounded away from $0$.
\end{itemize}
Overall, by said lemma (applied once for the sub-interval $I'$ and once for $I''$), we infer that there exist some fixed $c,\varrho>0$ such that, for all sufficiently large $N$,
\begin{equation}\label{eq:E-HKS4.2}
\P_{1,J_{\ell+1};\infty,0}^{b,1;g,g} \Big(\inf_{x \in I'} \{\phi(x)\} \vee \inf_{x \in I''} \{ \phi(x) \} > \varrho N^{1/3}\Big) \leq \exp(-c m^2)
\end{equation}
(as $\mathsf{w}_m = \Omega(m^2 N^{2/3})$ while 
% instead of $N^{2/3} \sqrt{m}$, 
the error probability in \cite[Lemma~4.2]{HKS25} is,  in our notation, 
$\exp(-c_1 \mathsf{w}_m N^{-2/3} + c_2 g^{3/2} /N)$).  With this in mind,  define for $x' \in I'$ 
and $y' \in \{0,\ldots,\varrho N^{1/3}\}$ the disjoint events
\begin{align*}
\cC_{x',y'} &:= \bigg\{ \phi(x') = y',  \phi(x) > \varrho N^{1/3},  \forall x \in I',  x > x' \bigg\}\,,
\end{align*}
and the analogous $\cC_{x'',y''}$ for $I''$ (i.e., $\phi(x'')=y''$, $\phi(x)>\varrho N^{1/3}$ for all $x<x'' $ in $I''$).  In view of \cref{eq:E-HKS4.2} our analysis 
reduces to one of the events $\cC_{x',y'}\cap \cC_{x'',y''}$,  for some $x',x'',y',y''$ where by the Gibbs--Markov
property,  the probability in \cref{eq:single-curve-prob-reduction} is at most
\begin{align}\label{eq:Eyal2}
    e^{-c m^2} + \max_{x',x'',y',y''}\P_{1,[x',x''];\infty,0}^{b,1;y',y''} \Big(\sup_{x \in [x',x'']}\{  \phi(x) \} > m' N^{1/3} \Big) 
\,.
\end{align}
At this point,  we appeal to \cite[Cor.~5.2]{HKS25} on the interval $I = [x',x'']$,  
with a large absolute constant $\mathsf{R}$ as specified there.  Indeed,  this requires 
an interval length $|I|$ exceeding $\sqrt{\mathsf{R}} (N/b)^{2/3}$,  while $|I| \leq \exp(\mathsf{R}^{-1/2} (N/b)^{1/3})$.
In our application,  $J_{\ell+1} \supset I \supset J_\ell$ hence $|I| \in [2\mathsf{w}_m, 2(K+1)\mathsf{w}_m]$ 
and our choice of $\mathsf{w}_m$ guarantees that the preceding requirements from $|I|$ hold for 
all large enough $m$.
The conclusion of the corollary is that,  for some $c>0$ and $C<\infty$,
\[ \P_{1,[x',x''];\infty,0}^{b,1;y',y''} \Big(\sup_{x \in [x',x'']}\{  \phi(x) \} > \mathsf{R}  (2 (\log m)^{3/2}+\varrho)  (N/b)^{1/3} \Big) \leq C m^{-c \mathsf{R}^{3/2}} \,.\]
With $m'=(\log m)^2$,  having $\mathsf{R}$ large enough guarantees a total error probability in \cref{eq:Eyal2}  
of at most $m^{-10}/K$ (uniformly over $x',x'',y',y''$).  This
completes the proof of \cref{eq:endpoints-N-13-induction} and hence also of \cref{eq:endpoints-N-13-top}.
\end{proof}

Moving to our second step,  note that similarly to \cref{rem:Qhat},  given  
% $(\underline{u},\underline{v})$ 
$(\widetilde{\phi}_r,\ldots,\widetilde{\phi}_K)$ at 
$\{a_{r-1}^-+2\Delta,a_r^--2\Delta\}$ for $1 \le r \le K$,
the restriction of $\widehat{\Q}^{(-)}_{\fh}$ to the
union of intervals
$I'_{-r} := [a_{r-1}^-+2\Delta,a_r^- - 2 \Delta]$
%  \subset I_{-r}$ 
forms an independent collection of ordered,  non-negative 
curves $\{\widetilde{\phi}_k(t),  r \le k \le K,  t \in I'_{-r} \}$
from such a \abbr{rw} \abbr{gate} on $I'_{-r}$,  with the 
specified boundary values at $\partial I'_{-r}$.  
%Having from \cref{lem:N-third-tight} the tightness of such rescaled by $N^{1/3}$ boundary values,  
It thus suffices to fix $1 \le r \le K$ and establish the claimed weak convergence 
(subject to shifts) to $\underline{X}^{(r)}$ of \cref{thm:1}, \cref{part:b},  
for $\underline{\psi}^{(-r)}_N$ based on  such \abbr{rw} \abbr{gate}
on $I'_{-r}$ with specified boundary values for which the event on the \abbr{lhs} of 
\cref{eq:endpoints-N-13-top} holds.   This follows upon applying
\cite[Thm.~1.2]{Serio23} in the setting of \cite[Rem.~1.4]{Serio23} with the following modifications.
% Serio Theorem 1.2 with $\lambda_N=N^{-1}$, $H_{\lambda_N}=N^{1/3}$, 
% $h_{\lambda_N}=N^{-1/3}$, $a_N=N^{1/3}$
While \cite{Serio23} has the lines restricted to the 
open Weyl chamber, namely, with strict inequalities in \cref{eq:1d-sos-domain-many-phi},
the same proof applies in our setting (of the closed Weyl chamber).
Also, while \cite[Thm.~1.2]{Serio23} is stated with given boundary 
values at times $T_\pm := \pm N$ and for the weak convergence of the law restricted to
the symmetric intervals $[-T N^{2/3}, T N^{2/3}]$, in view of 
\cite[Thm.~1.7]{Serio23}, the same applies for any choice of $T_\pm$ and
intervals of length $2T N^{2/3}$, whose distances from $T_\pm$, rescaled 
by $N^{2/3}$, still grow to infinity (as done in our 
selection of $t_N \in I_N^{(-r)}$).
%where $\mu_{K-|r|+1}^{\mathfrak{o}}$ denotes the law of the $\lambda$-tilted line ensemble with $(K-|r|+1)$ curves and parameter
%$\fa = \frac{\theta_\star}{\alpha_r}\sqrt{\Lambda''(\theta_\star)}
This completes the proof of \cref{thm:1}, \cref{part:b}.

\bibliographystyle{abbrv}
\bibliography{sos_wedge}

\end{document}